\documentclass[11pt,twoside]{article}
\usepackage[utf8]{inputenc}
\usepackage[T1]{fontenc}
\usepackage{latexsym}
\usepackage{amssymb,amsbsy,amsmath,amsfonts,amssymb,amscd}
\usepackage{graphicx,color}
\usepackage{float,url}
\usepackage{cancel}
\usepackage[colorlinks=true]{hyperref}
\usepackage{adjustbox}
\usepackage{comment}

\newcommand{\commentout}[1]{}
\newcommand{\AP}[1] {{#1}}
\newcommand{\R}{\mathbb{R}}
\newcommand{\N}{\mathbb{N}}

\newcommand {\e}  {\varepsilon}

\newcommand {\sg} {\sigma}
\newcommand {\vp} {\varphi}

\newcommand {\Chi} {{\bf \raise 2pt \hbox{$\chi$}} }
\newcommand {\cae} { {\mathcal E} }

\newcommand {\f}   {\frac}
\newcommand {\p}   {\partial}

\newcommand*{\dd}{\mathop{\kern0pt\mathrm{d}}\!{}}

\newcommand{\beq}{\begin{equation}}
\newcommand{\eeq}{\end{equation}}
\newcommand{\bea} {\begin{array}{rl}}
\newcommand{\eea} {\end{array}}
\newcommand{\bepa}{\left\{ \begin{array}{l}}
\newcommand{\eepa} {\end{array}\right.}
\newtheorem{theorem}{Theorem}%[section]

\newtheorem{proposition}[theorem]{Proposition}

\newcommand{\qed}{{ \hfill
                       {\unskip\kern 6pt\penalty 500 \raise -2pt\hbox{\vrule\vbox to 6pt{\hrule width 6pt
                       \vfill\hrule}\vrule} \par}   }}
\title{Relaxation of the Cahn-Hilliard equation with singular single-well potential and degenerate mobility.}

\author{Beno\^ \i t Perthame\thanks{Sorbonne Universit\'{e}, CNRS,  Universit\'{e} de Paris, Inria, Laboratoire Jacques-Louis Lions, F-75005 Paris, France.} 
 \thanks{Email: Benoit.Perthame@sorbonne-universite.fr} 
 \footnotemark[4]
\and
Alexandre Poulain\footnotemark[1] \thanks{Email: poulain@ljll.math.upmc.fr} 
 \thanks{The authors have received funding from the European Research Council (ERC) under the European Union's Horizon 2020 research and innovation programme (grant agreement No 740623)}
}
\date{\today}

\begin{document}
\maketitle
\pagestyle{plain}
%\tableofcontents
\pagenumbering{arabic}

\begin{abstract} 
  The degenerate Cahn-Hilliard equation is a standard model to describe living tissues. It takes into account cell populations undergoing short-range attraction and long-range repulsion effects. In this framework, we consider the usual Cahn-Hilliard equation with a singular single-well potential and degenerate mobility. These degeneracy and singularity induce numerous difficulties, in particular for its numerical simulation. To overcome these issues, we propose a relaxation system formed of two second order equations which can be solved with standard packages. This system is endowed with an energy and an entropy structure compatible with the limiting equation. Here, we study the theoretical properties of this system; global existence and convergence of the relaxed system to the degenerate Cahn-Hilliard equation. We also study the long-time asymptotics which interest relies on the numerous possible steady states with given mass. 
\end{abstract} 
\vskip .7cm

\noindent{\makebox[1in]\hrulefill}\newline
2010 \textit{Mathematics Subject Classification.}  35B40; 35G20 ; 35Q92; 92C10
\newline\textit{Keywords and phrases.} Degenerate Cahn-Hilliard equation;  Relaxation method; Asymptotic analysis; Living tissues
%
%%%%%%%%%%%%%%%%%%%%%%%%%%%%%%%%%%%%%%%%%%%%

\section{Introduction}
\label{sec:intro}
%-------------------------------------------
%%%%%%%%%%%%%%%%%%%%%%%%%%%%%%%%%%%%%%%%%%%%
% The model 
The Degenerate Cahn-Hilliard equation (DCH in short) is a standard model, widely used in the mechanics of living  tissues,  \cite{Benamar_Goriely_2005, wise_three-dimensional_2008, BenAmar_C_F, agosti_cahn-hilliard-type_2017, agosti_self-organised_nodate,frigeri_multi-species_2018}. It is usual to set this problem in a  smooth bounded domain $\Omega  \subset \mathbb{R}^d$ with the zero flux boundary condition
\begin{equation}
\p_t n = \nabla \cdot \left( b(n) \nabla \left( -\gamma \Delta n + \psi^\prime(n) \right) \right)  \quad \text{in} \quad  \Omega \times (0,+\infty),
\label{eq:CH}
\end{equation}
\begin{equation}
\f{\p n}{\p \nu} = b(n)\f{ \p \left(- \gamma \Delta n + \psi^\prime(n) \right)}{\p \nu} = 0 \qquad \text{on} \quad  \p \Omega \times (0,+\infty),
\label{eq:CH-bound}
\end{equation}
where $\nu$ is the outward normal  vector to the boundary $\p \Omega$ and $n = \frac{n_1}{n_1+n_2}$ represents the relative density or volume fraction of one of the two cell types.

Degeneracy of the coefficient $b(n)$ and singularity of  the potential $\psi(n)$ make this problem particularly difficult to solve numerically and in particular, to preserve the apriori bound $0\leq n <1$ . Motivated by the use of standard software for elliptic or parabolic equations, we propose to study the following relaxed degenerate Cahn-Hilliard equation (RDHC in short)
\begin{equation}\left\{
\begin{aligned}
\p_t n &= \nabla \cdot \left(b(n) \nabla \left(\vp + \psi_+^\prime(n) \right) \right) \quad &\text{ in } \Omega \times (0,+\infty),
\\
- \sg \Delta \vp + \vp &= -\gamma \Delta n + \psi_-^\prime \big(n-\f{\sg}{\gamma}\vp \big)\quad &\text{ in } \Omega \times (0,+\infty).
\end{aligned}\right.
\label{eq:CH-relax}
\end{equation}
supplemented with zero-flux boundary conditions
\begin{equation}
\f{\p (  \gamma n- \sg \vp)}{\p \nu} =  b(n) \f{\p \big( \vp + \psi_{+}^\prime(n) \big)}{\p \nu} = 0 \qquad \text{ on } \p \Omega \times (0,+\infty).
\label{eq:CH-bound-relax}
\end{equation}
Our purpose is to study existence for this system, to prove that as $\sg \to 0$, the solution of RDCH system converges to the solution of the DCH equation and study the possible long term limits to steady states.
\\

We make the following assumptions for the different inputs of the system~\eqref{eq:CH-relax}.
For the mechanics of living tissues, the usual assumption is that the potential  $\psi$ is concave degenerate near $n=0$ (short-range attraction) and convex for $n$ not too small (long-range repulsion). Additionally, a singularity at $n=1$ is desired  to represent saturation by one phase \cite{byrne_modelling_2004}. For these reasons, we call the potential {\em single-well logarithmic}  and we decompose it in a convex and a concave part $\psi_\pm$ 
\begin{equation}
\psi(n) = \psi_+(n) + \psi_-(n),  \qquad \pm  \psi_\pm''(n) \geq 0, \qquad 0\leq n < 1.
\label{eq:psi-dec}
\end{equation}   
The singularity is contained in the convex part of the potential and we assume that 
\begin{equation}
\psi_+ \in C^2\big([0,1) \big), \quad \psi'_+(1)=\infty,
\label{eq:psi-plus}
\end{equation}
and  we extend the smooth concave part on $[0, 1]$ to the full line with 
\begin{equation}
\psi_- \in C^2(\R) \qquad \psi_-, \; \psi^\prime_-,\;  \psi''_- \quad  \text{are bounded and } \f{\sg}{\gamma} ||\psi_-^{\prime\prime}||_\infty < 1.
\label{eq:psi-minus}
\end{equation}
In practice, typical examples of potentials are, for some $n^* \in (0,1)$, see \cite{colombo_towards_2015,chatelain_morphological_2011} 
\begin{equation}
    \psi(n)= -(1-n^*)\ln(1-n) - \frac{n^3}{3} -(1-n^*)\frac{n^2}{2} -(1-n^*)n + k,
    \label{eq:pot1}
\end{equation}
\begin{equation}
    \psi(n)=  \frac 1 2 n \ln n + (1-n) \ln (1-n) - (n-\frac 1 2 )^2.
    \label{eq:pot2}
\end{equation}
The potential \eqref{eq:pot1} fulfills our assumptions and the convex/concave decomposition reads for $n\in [0, 1)$
\[
    \psi_+(n) = -(1-n^*)\log(1-n) - \frac{n^3}{3}, \qquad \psi_-(n) = -(1-n^*)\frac{n^2}{2} -(1-n^*)n + k.
\]
In this case $\psi_+$ is convex if $n^*\le 0.7$.
Potential \eqref{eq:pot2} does not satisfy our assumptions because of the additional singularity at $0$ (and thus is not treated here), however, it can also be decomposed as needed with
\[
    \psi_+(n) = \frac 1 2 n \ln n + (1-n) \ln (1-n) , \qquad \psi_-(n) = - (n-\frac 1 2 )^2.
\]
To satisfy the assumptions \eqref{eq:psi-plus} and \eqref{eq:psi-minus}, we need to extend the potential $\psi_-$ to all $\R$ since the above examples are defined for $n\in [0, 1)$, which is an immediate task.

The potential \eqref{eq:pot1} has been used to model the interaction between cancer cells from a glioblastoma multiforme and healthy cells by Agosti \textit{et al.} \cite{agosti_computational_2018} and promising results have been obtained. 
We also use the degeneracy assumption on $b\in C^1([0,1]; \R^+)$,
\begin{equation}
b(0)=b(1)= 0,  \qquad b(n) >0 \text{ for }0<n< 1. 
\label{eq:assb}
\end{equation}   

The typical expression in the applications we have in mind is $b(n)=n (1-n)^2$. Consequently, when considered as transport equations, both~\eqref{eq:CH} and~\eqref{eq:CH-relax} impose formally the property that $0 \leq n \leq 1$. However, we need an additional technical assumption, namely that there is some cancellation at $1$ such that 
\begin{equation}
b(\cdot) \psi^{\prime\prime}(\cdot) \in C([0,1]; \R).
\label{eq:assbpsi}
\end{equation}   

We implicitly assume~\eqref{eq:psi-dec}--\eqref{eq:assbpsi} in this paper.  Also, we always impose an initial condition satisfying
\begin{equation}
n^0 \in H^1(\Omega), \qquad 0 \le n^0 < 1 \quad \text{a.e. in }  \Omega.
\end{equation}
The assumption $n^0\in [0,1)$ is consistent with the degeneracy of mobility at $0$ which allows solutions to vanish on open sets. But the singularity of the potential at $1$ and the energy bound make that $n=1$ cannot be achieved except of a negligible set.
Thanks to the boundary condition~\eqref{eq:CH-bound}, the system conserves the initial mass 
\[
\int_\Omega n(x,t) dx =\int_\Omega n^0(x) dx =: M, \quad \forall t \ge 0.
\] 
We denote the flux associated with the RDCH system by
\begin{equation}
J_\sg(n,\vp) := -b(n)\nabla\left(\vp+\psi^\prime_+(n) \right).
\end{equation}

The system~\eqref{eq:CH-relax} comes with energy and entropy structures, namely, the energy is defined as
\begin{equation}
\cae_{\sigma} [n_{\sg}] = \int_\Omega \left[ \psi_{+}(n_{\sg})  +  \f{\gamma}{2}|\nabla (n_{\sg}-\f{\sigma}{\gamma}\vp_{\sg})|^2 + \f{\sg}{2 \gamma} |\vp_{\sg}|^2 + \psi_-(n_{\sg} -\f{\sg}{\gamma}\vp_{\sg}) \right] .
\label{eq:energy_L}
\end{equation}
The energy is bounded from below thanks to the assumptions above and satisfies 
\begin{equation}
\f{d}{dt} \cae_{\sigma} [n_{\sg}(t)]  = -\int_\Omega b(n_{\sg}) \big|\nabla (\vp_{\sg} + \psi'_{+}(n_{\sg}))\big|^2 \leq 0 .
\label{eq:deriv-energy_L}
\end{equation}
For the entropy, we set for $0< n <1$ the singular function 
 \begin{equation}
\phi^{\prime \prime}(n) = \f{1}{b(n)}, \qquad  \Phi[n] = \int_\Omega \phi \big(n(x)\big)dx.
\label{eq:assumption-entropy_L}
\end{equation}
The entropy functional behaves as follows in the case  $b(n)= n (1-n)^2$
\[
\phi (n)= n \log(n), \; n \approx 0^+, \qquad \phi (n)= - \log(1-n), \; n \approx 1^-.
\]
The relation holds
\begin{equation}
\begin{aligned}
\f{ d \Phi[n_{\sg}(t)]}{dt} = -\int_\Omega \gamma \left| \Delta\left( n_{\sg} - \f{\sg}{\gamma} \vp_{\sg}\right)\right|^2 + \f{\sg}{\gamma}|\nabla \vp_{\sg}|^2 
&+ \psi''_-(n_{\sg} - \f{\sg}{\gamma} \vp_{\sg}) \left| \nabla(n_{\sg} - \f{\sg}{\gamma}\vp_{\sg}) \right|^2 \\
&+ \psi^{\prime\prime}_{+}(n_{\sg})|\nabla n_{\sg}|^2.
\end{aligned}
\label{eq:entropy_L}
\end{equation}
%--------------------------
Notice that entropy equality does not provide us with a direct a priori estimate because of the term $\psi^{\prime \prime}_-$ can be negative.  Therefore we have to combine it with the energy dissipation to write  
\[
\begin{aligned}
\Phi [n_{\sg}(T)] + \int_{\Omega_T} & \left[ \gamma \left| \Delta\left( n_{\sg} - \f{\sg}{\gamma} \vp_{\sg}\right)\right|^2 + \f{\sg}{\gamma}|\nabla \vp_{\sg}|^2 + \psi^{\prime\prime}_{+}(n_{\sg})|\nabla n_{\sg}|^2\right] 
\\
&\leq  \Phi[n^0] + \frac{2T}{\gamma} \| \psi''_- \|_\infty \;  \cae_{\sg}[n^0].
\end{aligned}
\]

% Paragraph: other systems
The first use of the Cahn-Hilliard equation is to  model the spinodal decomposition occurring in binary materials during a sudden cooling~ \cite{cahn_free_1958,cahn_spinodal_1961}. The bilaplacian $-\gamma \Delta^2 n$ is used to represent surface tension and the parameter $\gamma$ is the square of the width of the diffuse interface between the two phases.  In both equations~\eqref{eq:CH} and~\eqref{eq:CH-relax}, $n = n(x,t)$ is a relative quantity: for our biological application this represents a relative cell density as derived from phase-field models~\cite{byrne_modelling_2004} and for this reason the  property  $n\in [0,1)$  is relevant. The biological explanation of the fact that $1$ is excluded from the interval of definition of $n$ is due to the observation that cells tend to not form aggregates that are too dense. For instance, the two phases can be  the relative density of cancer cells and the other component represents the extracellular matrix, liquid, and other cells. This  binary mixture tends to form aggregates in which the density of one component of the binary mixture is larger than the other component. The interest of the Cahn-Hilliard equation stems from  solutions that reproduce  the formation of such clusters of cells {\em in vivo} or on dishes.  
Several variants are also used. A Cahn-Hilliard-Hele-Shaw model is proposed by Lowengrub {\em et al}~\cite{lowengrub_analysis_2013} to describe the avascular, vascular and metastatic stages of solid tumor growth. They proved the existence and uniqueness of a strong solution globally for $d\le 2$ and locally for $d=3$ as well as the long term convergence to steady-state. The case with a singular potential is treated in~\cite{GIGr2018}. Variants can include the coupling with fluid equations and chemotaxis, see for instance~\cite{EbeGarcke2019} and the references therein.
\par

% Paragraph: long term
The analysis of the long-time behavior of the solution of the Cahn-Hilliard equation has also attracted much attention since the seminal paper~\cite{ElliottBlowey91}. A precise description of the $\omega$-limit set has been obtained in one dimension for the case of smooth polynomial potential and constant mobility in \cite{songmu_asymptotic_1986}. In this work, the effect of the different parameters of the model such as the initial mass, the width of the diffuse interface are investigated. In fact, the authors  show that when $\gamma$ is large, the solution converges to a constant  as $t\to \infty$.
The same happens when the initial mass is large. However when $\gamma$ is positive and small enough, the system admits nontrivial steady-states. For logarithmic potentials and constant mobility, Abels and Wilke~\cite{abels_convergence_2007} prove that solutions converge to a steady-state as time goes to infinity using the Lojasiewicz–Simon inequality. Other works have been made on the long term behavior of the solutions of some Cahn-Hilliard models including a source term~\cite{cherfils_generalized_2014}, with dynamic boundary conditions~\cite{gilardi_long_2010}, coupled with the Navier-Stokes equation~\cite{gal_asymptotic_2010}, for non-local interactions and a reaction term~\cite{iuorio_long-time_2017}. 
\par 

% Paragraph: numerics
Many difficulties, both analytical and numerical,  arise in the context of Cahn-Hilliard  equation and its variants. Because of the bilaplacian term, most of the numerical methods require to change the equation \eqref{eq:CH} into a system of two coupled equations
\begin{equation}\left\{
\begin{aligned}
\p_t n &= \nabla \cdot \left(b(n) \nabla v  \right),\\
v &= -\gamma \Delta n + \psi^\prime(n).
\end{aligned}\right.
\label{eq:CH-coupled}
\end{equation}
%This system~\eqref{eq:CH-coupled} induces difficulties because the second equation contains all the backward diffusion and without a regularization of the unknown $v$, this system will lead to an ill-posed problem.
This system of equations has been analyzed in the case where the mobility is degenerate and the potential is a logarithmic double-well functional by Elliott and Garcke \cite{elliott_cahn-hilliard_1996}. They establish the existence of weak solutions of this system. Agosti \textit{et al} \cite{agosti_cahn-hilliard-type_2017} establish the existence of weak solutions when $\psi$ is a single-well logarithmic potential which is more relevant for biological applications (see \cite{byrne_modelling_2004}). They also prove that this system preserves the positivity of the cell density and the weak solutions belong to 
\[
n\in L^\infty(0,T;H^1(\Omega)) \cap L^2(0,T;H^2(\Omega)) \cap H^1(0,T;(H^1(\Omega))^\prime),  \quad J \in L^2( (0,T) \times \Omega, \R^d) \quad
 \forall T>0, 
\] 
\par
The Cahn-Hilliard equation can be seen as an approximation of  the famous  microscopic model in \cite{GL1, GL2}. With our notations, it reads
\[ 
\p_t n = \nabla \cdot \left[b(n) \nabla \big( K_\sg \star n+\psi'(n)  \big) \right],
\]
with a symmetric smooth kernel $K_\sg  \underset{ \sg \to 0 }{\longrightarrow} \Delta \delta$. The convergence to the DCH equation has been answered recently in \cite{davoli_degenerate_2019} in the case of periodic boundary conditions. Although, very similar in its form, our relaxation model undergoes different a priori estimates which allow us to study differently the limit $\sg \to 0$ for \eqref{eq:CH-relax}.
\par
For a full review about the mathematical analysis of the  Cahn-Hilliard equation and its variants, we refer the reader to the recent book of Miranville \cite{miranville_cahn?hilliard_2019}.

Numerical simulations of the DCH system have been also performed in the context of double-well potentials in~\cite{elliott_cahn-hilliard_1986,barrett_finite_1999}. To keep the energy inequality is a major concern in numerical methods and the survey paper by Shen {\em et al}~\cite{ShenJieXY2019} presents a general method applied to the present context. 

% Paragraph: back to our problem
\par Numerics is also our motivation to propose a relaxation of equation~\eqref{eq:CH}  in a form close to the writing~\eqref{eq:CH-coupled}. We recover the system~\eqref{eq:CH-relax} by introducing a new potential $\vp$ and a regularizing equation which defines  $v$ through $\nabla \vp$. We use the decomposition~\eqref{eq:psi-dec}  of the potential to keep the convex and stable part in the main equation for $n$, rejecting the concave and unstable part in the regularized equation. The relaxation parameter is  $\sigma$ and we need to verify that, in the limit $\sg \rightarrow 0$, we recover the original DCH equation~\eqref{eq:CH}.  This is the main purpose of the present paper. 
\par

As a first step towards the existence of solutions of~\eqref{eq:CH-relax}, in section \ref{sec:reg-ineq-exi}, we introduce a regularized problem which is not anymore degenerate and we prove the existence of weak solutions for this regularized-relaxed Cahn-Hilliard system. We show energy and entropy estimates from which we obtain a priori estimates which are used later on.
In section~\ref{sec:existence}, we pass to the limit in the regularization parameter $\epsilon$ and show the existence of weak solutions of the RDCH system. 
Then, in section~\ref{sec:cvg}, we prove the convergence as $\sigma \to 0$ to the full DCH model. Section~\ref{sec:ltb} is dedicated to the study of the long term convergence of the solutions to steady-states. We end the paper with some conclusions and perspectives.

%\input{tex/regularized}
%%%%%%%%%%%%%%%%%
\section{The regularized problem}
\label{sec:reg-ineq-exi}
%----------------------------------------------

To prove that the system \eqref{eq:CH-relax}, admits solutions and to precise the functional  spaces, we first define a regularized problem. Then we prove the existence of solutions and estimates based on energy and entropy relations.

%--------------------------------------------------
\subsection{Regularization procedure}
%--------------------------------------------------

We consider a small positive parameter $0 < \epsilon \ll 1$ and define the regularized mobility
\begin{equation}
B_\epsilon(n) = \begin{cases}
	b(1-\epsilon) &\text{ for } n \ge 1-\epsilon, \\
	b(\epsilon) &\text{ for } n \le \epsilon, \\
	b(n) &\text{ otherwise}.
\end{cases}
\label{eq:reg_mob}
\end{equation}
Then, there are  two positive constants $b_1$ and $B_1$,  such that
\begin{equation} 
	b_1 < B_\epsilon(n) < B_1,\quad \forall n \in \R. 
\label{eq:bn-assumption}	
\end{equation}
Thus, the regularized mobility satisfies
\begin{equation}
	B_\epsilon \in C(\R,\R^+).
	\label{eq:reg-mob-continuity}
\end{equation}
To define a regular potential, we smooth out the singularity located at $n=1$ which only occurs  in $\psi_+$, see~\eqref{eq:psi-plus}--\eqref{eq:psi-minus},  and preserve the assumption \eqref{eq:assbpsi} by setting 
\begin{equation}
\psi_{+,\epsilon}^{\prime \prime}(n) =
\begin{cases}
 \psi_+^{\prime \prime}(1-\epsilon) &\text{ for } n \ge 1-\epsilon, \\[5pt]
  \psi_+^{\prime \prime}(\epsilon) &\text{ for } n \le \epsilon ,   \\[5pt]
   \psi_+^{\prime \prime}(n)  &\text{ otherwise}.
\end{cases}
\end{equation}
%Then, we set, 
%\[\psi_\epsilon(n) = \psi_{+,\epsilon}(n) + \psi_-(n), \]
It is useful to notice that, for some positive constants $D_1$ independent of $0< \epsilon \leq \epsilon_0$ and $D_\epsilon$, we have  
\begin{equation}
\psi_{+,\epsilon}(n) \in C^2(\R, \R)\quad  \psi_{+,\epsilon}(n) \geq -D_1, \quad  |\psi_\epsilon^{\prime} (n)|\le D_\epsilon(1+|n|) , \quad \forall n \in \R .
\label{eq:assume-psi}
\end{equation}
See also \cite{agosti_cahn-hilliard-type_2017} for details about the extensions needed for the potential~\eqref{eq:pot1}. 

\par
We can now define the regularized problem
\begin{equation}
\left\{
\begin{aligned}
\p_t n_{\sg,\epsilon} &= \nabla \cdot \left[ B_\epsilon(n_{\sg,\epsilon}) \nabla(\vp_{\sg,\epsilon} + \psi_{+,\epsilon}^\prime(n_{\sg,\epsilon}) ) \right], \\
- \sg \Delta  \vp_{\sg,\epsilon} + \vp_{\sg,\epsilon} &= -\gamma\Delta n_{\sg,\epsilon} + \psi_-^\prime(n_{\sg,\epsilon} - \f \sg \gamma \vp_{\sg,\epsilon}),
\end{aligned}
\right.
\label{eq:reg_pb}
\end{equation}
with zero-flux boundary conditions
\begin{equation}
\f{\p (n_{\sg,\epsilon}-\f{\sigma}{\gamma}\vp_{\sg,\epsilon}) }{\p \nu} = \f{\p \big(\vp_{\sg,\epsilon} + \psi_{+,\epsilon}^\prime(n_{\sg,\epsilon}) \big)}{\p \nu} = 0 \qquad \text{ on } \; \p \Omega \times (0,+\infty).
\label{eq:bound_reg}
\end{equation}
It is convenient to define  the flux of the regularized system as
\[
J_{\sigma,\epsilon} = - B_\epsilon(n_{\sg,\epsilon})\nabla \left(\vp_{\sg,\epsilon} + \psi_{+,\epsilon}^\prime(n_{\sg,\epsilon}) \right).
\] 

%----------------------------------------------------------------
%----------------------------------------------------------------
\subsection{Existence for the regularized problem}
%--------------------------------------------------------------

We can now state the existence  theorem  for the regularized problem~\eqref{eq:reg_pb}.
\begin{theorem} [Existence for $\e >0$]
    Assuming $n^0 \in H^1(\Omega)$, there exists a pair of functions $(n_{\sigma,\epsilon},\vp_{\sigma,\epsilon})$ such that for all $T>0$,
    \begin{align*}   
        n_{\sigma,\epsilon} &\in L^2(0,T;H^1(\Omega)), \qquad 
        \partial_t n_{\sigma,\epsilon} \in  L^2(0,T;(H^1(\Omega))') ,\\ 
%        n_{\sigma,\epsilon}(0) &= n^0 \text{ and } \nabla n_{\sigma,\epsilon}\cdot \nu = 0 \text{ on } \partial \Omega ,\\
        \vp_{\sigma,\epsilon} &\in L^2(0,T;H^1(\Omega)), \\
        n_{\sigma,\epsilon}-\f\sg\gamma \vp_{\sigma,\epsilon} &\in L^2(0,T;H^2(\Omega)), \qquad \p_t \left(n_{\sigma,\epsilon}-\f\sg\gamma \vp_{\sigma,\epsilon}\right) \in  L^2(0,T;(H^1(\Omega))') ,
    \end{align*}  
 which satisfies the regularized-relaxed degenerate Cahn-Hilliard equation~\eqref{eq:reg_pb}, \eqref{eq:bound_reg} in the following weak sense: for all test function $\chi \in L^2(0,T;H^1(\Omega))$, it holds
    \begin{equation}
        \begin{aligned}
            \int_0^T <\chi, \p_t n_{\sigma,\epsilon}>  &= \int_{\Omega_T} B_\epsilon(n_{\sigma,\epsilon}) \nabla\left(\vp_{\sigma,\epsilon} + \psi^\prime_{+,\epsilon}(n_{\sigma,\epsilon}) \right) \nabla \chi ,\\
            \sigma \int_{\Omega_T} \nabla \vp_{\sigma,\epsilon} \nabla \chi &+ \int_{\Omega_T} \vp_{\sigma,\epsilon} \chi = \gamma\int_{\Omega_T} \nabla n_{\sigma,\epsilon} \nabla \chi + \int_{\Omega_T} \psi_-^\prime(n_{\sigma,\epsilon}-\frac{\sigma}{\gamma}\vp_{\sigma,\epsilon}) \chi  .
        \end{aligned}
        \label{eq:limit-reg-rel-sys}
    \end{equation}
    \label{th:existence-regularized}
\end{theorem}
%----------------------------------------
\begin{proof}
   We adapt the proof of the theorem $2$ in \cite{elliott_cahn-hilliard_1996} where the authors prove the existence of solutions of the Cahn-Hilliard system with positive mobilities. Since the regularized mobility here is positive due to \eqref{eq:bn-assumption}, we can apply the same theorem. 
    The proof of existence follows the following different stages 
\\[5pt]
\textit{Step 1. Galerkin approximation.} Firstly, we make an approximation of the regularized problem \eqref{eq:reg_pb}. We define the family of  eigenfunctions $\{\phi_i\}_{i\in \N}$ of the Laplace operator subjected to zero Neumann boundary conditions.
        \[
            -\Delta \phi_i = \lambda_i\phi_i \text{ in } \Omega\quad  \text{ with } \quad  \nabla \phi_i \cdot \nu = 0 \text{ on } \partial \Omega. 
        \]
The family   $\{\phi_i\}_{i\in \N}$ form an orthogonal basis of both $H^1(\Omega)$ and $L^2(\Omega)$ and we normalize them, i.e. $(\phi_i,\phi_j)_{L^2(\Omega)} = \delta_{ij}$ to obtain an orthonormal basis. We assume that the first eigenvalue is $\lambda_1 =0$ (which does not introduce a lack of generality). 
        
We consider the following discretization of \eqref{eq:reg_pb}
        \begin{align}
            n^N(t,x) &= \sum_{i=1}^N c_i^N(t) \phi_i(x),\quad \vp^N(t,x) = \sum_{i=1}^N d_i^N(t) \phi_i(x), \label{eq:semi-discr1} \\
            \int_\Omega \partial_t n^N \phi_j &= -\int_\Omega B_\epsilon(n^N) \nabla\left(\vp^N + \Pi^N\left(\psi_{+,\epsilon}^\prime(n^N)\right) \right) \nabla \phi_j, \quad \text{for } j=1,...,N, \label{eq:semi-discr2}\\
            \int_\Omega \vp^N \phi_j &= \gamma \int_\Omega \nabla\left(n^N - \frac{\sigma}{\gamma} \vp^N \right)\nabla \phi_j + \int_\Omega \psi_-^\prime(n^N- \frac{\sigma}{\gamma} \vp^N ) \phi_j, \quad \text{for } j=1,...,N, \label{eq:semi-discr3}\\
            n^N(0,x) &= \sum_{i=1}^N \left(n_0,\phi_i \right)_{L^2(\Omega)}\phi_i. \label{eq:semi-discr4}
        \end{align}
We have used the $L^2$ projection $\Pi^N: L^2(\Omega) \to V$, where $V= \text{span}\{\phi_1,...,\phi_N \}$.
This gives the following initial value problem for a system of ordinary differential equations, for all $j=1,...,N$, 
        \begin{align}
            \partial_t c^N_j &= -  \int_\Omega B_\epsilon(\sum_{i=1}^N c^N_i \phi_i) \nabla \left( \vp^N+ \Pi^N\left(\psi_{+,\epsilon}^{\prime}(\sum_{i=1}^N c^N_i \phi_i)\right) \right) \nabla \phi_j   , \label{eq:init-val1}  \\
            d^N_j &= \gamma \lambda_j c^N_j - \sigma \lambda_j d^N_j + \int_\Omega \psi_-^\prime(\sum_{k=1}^N (c^N_k - \frac{\sigma}{\gamma} d^N_k)\phi_k) \phi_j, \label{eq:init-val2}  \\
            c_j^N(0) &= \left(n_0,\phi_j \right)_{L^2(\Omega)}.
 \label{eq:init-val3} 
\end{align}
 Since the right-hand side of equation \eqref{eq:init-val1} depends continuously on the coefficients $c^N_j$, the initial value problem has a local solution.
 \\[5pt]       
 %----------------------------------------------------------
 \textit{ Step 2. Inequalities and convergences.}
Multiplying equation \eqref{eq:init-val1}, by $\phi_i \big(\vp^N+\psi_+^\prime(n^N) \big)$, then summing over $i$ and integrating over the domain leads to
 \begin{equation}
    \begin{aligned}
            \f{d}{dt} \int_\Omega \psi_{+,\epsilon}(n^N)  &+ \int_\Omega \partial_t(n^N) \vp^N \\
    &= \int_\Omega \sum_i (\vp^N+ \psi^\prime_{+,\epsilon}(n^N))\phi_i \int_\Omega \nabla \phi_i \left(B_\epsilon(n^N) \nabla\left(\vp^N + \Pi^N\left(\psi_{+,\epsilon}^\prime(n^N)\right) \right)  \right) \, \dd y \, \dd x.  
\end{aligned} 
    \label{eq:energy-discrete-1}
 \end{equation}
 Let us focus on the left-hand side with
 \[
    \int_\Omega \partial_t(n^N) \vp^N  = \int_\Omega\partial_t(n^N-\f\sg\gamma \vp^N)\vp^N + \f12 \f \sg\gamma \f{d}{dt} \int_\Omega |\vp^N|^2.
 \] 
 Then, using the equation \eqref{eq:semi-discr3}, we have that
 \[
    \int_\Omega\partial_t(n^N-\f\sg\gamma \vp^N)\vp^N = \f\gamma 2 \f{d}{dt} \int_\Omega |\nabla (n^N-\f{\sg}{\gamma} \vp^N) |^2 + \f{d}{dt}\int_\Omega \psi_-(n^N - \f{\sg}{\gamma}\vp^N).
 \]
 The right-hand side of equation \eqref{eq:energy-discrete-1} gives 
\[
    \begin{aligned}
    -\int_\Omega \sum_i &(\vp^N+ \psi^\prime_{+,\epsilon}(n^N)) \phi_i \int_\Omega \nabla \phi_i \left(B_\epsilon(n^N) \nabla\left(\vp^N + \Pi^N\left(\psi_{+,\epsilon}^\prime(n^N)\right) \right)  \right) \, \dd y \, \dd x \\
    &=  - \int_\Omega B_\epsilon(n^N) \left|\nabla\left(\vp^N + \Pi^N\left(\psi_{+,\epsilon}^\prime(n^N) \right)\right)\right|^2.
    \end{aligned}
\]
Altogether, we obtain 
\begin{equation}
\f{d}{dt} E(t) + \int_\Omega B_\epsilon(n^N) \left|\nabla\left(\vp^N + \Pi^N\left( \psi_{+,\epsilon}^\prime(n^N)\right) \right)\right|^2   \le 0,
\end{equation}
where 
\[
E(t) = \int_\Omega \psi_{+,\epsilon}(n^N) + \f\gamma 2\int_\Omega |\nabla (n^N-\f{\sg}{\gamma} \vp^N) |^2 + \f12 \f \sg\gamma \int_\Omega |\vp^N|^2+\int_\Omega \psi_-(n^N - \f{\sg}{\gamma}\vp^N).
\]
Next, to prove the compactness in space of $\nabla n^N$, we write 
\[
    \min_{n^N} \left(\frac{1+\f\sg\gamma \psi_{+,\epsilon}^{\prime\prime}}{\psi_{+,\epsilon}^{\prime\prime}} \right)^2 \int_\Omega \left|\nabla \psi_{+,\epsilon}^{\prime}(n^N) \right|^2
     \le \int_{\Omega} \left(\frac{1+\f\sg\gamma \psi_{+,\epsilon}^{\prime\prime}}{\psi_{+,\epsilon}^{\prime\prime}} \right)^2 \left|\nabla \psi_{+,\epsilon}^{\prime}(n^N) \right|^2 \le \int_\Omega \left|\nabla \left(n^N + \f\sg\gamma \psi_{+,\epsilon}^{\prime}(n^N)  \right) \right|^2.
\]
Therefore, for some  $\theta >0$, we have
\[
    \begin{aligned}
    \left(\left(\f\sg\gamma \right)^2 +\theta\right) \int_\Omega \big|\nabla \psi_{+,\epsilon}^\prime(n^N)\big|^2 \le  &\int_\Omega \big|\nabla\left(n^N-\f\sg\gamma \vp^N \right) + \f\sg\gamma \nabla \left(\vp^N + \Pi^N\left(\psi_{+,\epsilon}^\prime(n^N)\right) \right) \\&+ \f\sg\gamma\nabla\left(\psi_{+,\epsilon}^\prime(n^N) - \Pi^N\left(\psi_{+,\epsilon}^\prime(n^N) \right) \right)  \big|^2  .
    \end{aligned}
\]
Finally, we obtain
\[
    \begin{aligned}
    \left(\left(\f\sg\gamma \right)^2 +\theta \right) \int_\Omega \big|\nabla \psi_{+,\epsilon}^\prime(n^N)\big|^2 &\le C(T) +  \left(\f\sg\gamma \right)^2 \int_\Omega \left|\nabla\left( \psi_{+,\epsilon}^\prime(n^N) - \Pi^N\left(\psi_{+,\epsilon}^\prime(n^N) \right)\right) \right|^2 \\
    & \le C(T) + \left(\f\sg\gamma \right)^2 \int_\Omega \left|\nabla \psi_{+,\epsilon}^\prime(n^N)  \right|^2,
    \end{aligned}
\]
and we proved that 
\[
      \theta \int_\Omega \big|\nabla \psi_{+,\epsilon}^\prime(n^N) \big|^2 \le C(T).
\]
    Therefore, we can obtain from the previous inequalities the following
        \begin{align}
            \f{\gamma}{2}\int_\Omega &|\nabla (n^N-\f{\sigma}{\gamma}\vp^N)|^2\le C, \label{eq:semi-discr-ineq1}\\
            \f{\sg}{2 \gamma} \int_\Omega &|\vp^N|^2 \le C, \label{eq:semi-discr-ineq2} \\
            \int_{\Omega_T} B_\epsilon(n^N) &\big| \nabla \left(\vp^N + \Pi^N\left(\psi^\prime_{+,\epsilon}(n^N)\right)\right) \big|^2 \le C, \label{eq:semi-discr-ineq3}  \\
            \theta \min_{r\in\R}\left(\psi_{+,\epsilon}^{\prime\prime}(r) \right) \int_\Omega &|\nabla n^N|^2 \le C(T), \label{eq:semi-discr-ineq4}
        \end{align}
        which hold for positive values of $\gamma, \sigma, \theta$ and also for all finite time $T\ge 0$. Therefore, from these inequalities we can extract subsequences of $(n^N,\vp^N)$ such that the following convergences hold for any time $T\ge 0$ and small positive values of $\gamma, \sigma$.\par
        Taking $j=1$ in \eqref{eq:semi-discr2}, gives the results that $\f{d}{dt}\int n^N = 0$. 
        Then, using the inequality \eqref{eq:semi-discr-ineq4} and the Poincaré-Wirtinger inequality, we obtain 
        \begin{equation}
            n^N \rightharpoonup n_{\sg,\epsilon} \text{ weakly in } L^2(0,T;H^1(\Omega)).
        \label{eq:semi-discr-n-weak}
        \end{equation} 
        This result, in turn,  implies that the coefficients $c^N_j$ are bounded and a global solution to ~\eqref{eq:init-val1}--\eqref{eq:init-val3} exists.
        Choosing $j=1$ in \eqref{eq:semi-discr3} gives
        \[
            \int_\Omega \vp^N  = \int_\Omega \psi_-\left(n^N-\f\sg\gamma n^N \right),
        \]
        and combining \eqref{eq:semi-discr-ineq1}, \eqref{eq:semi-discr-n-weak} and the Poincaré-Wirtinger  inequality gives
        \begin{equation}
            \vp^N \rightharpoonup \vp_{\sg,\epsilon} \text{ weakly in } L^2(0,T;H^1(\Omega)).
            \label{eq:semi-discr-conv-vp}
        \end{equation}
        We also obtain from \eqref{eq:semi-discr-n-weak} and \eqref{eq:semi-discr-conv-vp}  
        \begin{equation}
           n^N - \f \sg\gamma \vp^N \rightharpoonup n_{\sg,\epsilon}-\f \sg\gamma \vp_{\sg,\epsilon} \text{ weakly in } L^2(0,T;H^1(\Omega)).
           \label{eq:semi-discr-conv-n-vp}
        \end{equation}
        From the previous convergence, we conclude that $\vp_{\sg,\epsilon}\in L^2(0,T;H^1(\Omega))$, therefore, using elliptic regularity we know that 
        \begin{equation}
            n_{\sg,\epsilon}-\f\sg\gamma \vp_{\sg,\epsilon} \in L^2(0,T;H^2(\Omega)).
            \label{eq:H^2-n-vp}
        \end{equation} 
        
        To be able to prove some strong convergence in $L^2(0,T;L^2(\Omega))$ of $n^N$, we need an information about the temporal derivative $\partial_t n^N$. From the first equation of the system, we have for all test functions $\phi \in L^2(0,T;H^1(\Omega))$
         \begin{equation}
            \begin{aligned}
            \left|\int_{\Omega_T} \partial_t n^N \phi \right| &= \left|\int_{\Omega_T} \partial_t n^N \Pi_N \phi \right| \\
            &= \left|\int_{\Omega_T} b(n^N) \nabla\left(\vp^N + \Pi^N\left(\psi^\prime_{+,\epsilon}(n^N)\right) \right) \nabla \Pi_N \phi \right| \\
            & \le  \left(  B_1\int_{\Omega_T} B_\epsilon(n^N) \left| \nabla\left(\vp^N + \Pi^N\left(\psi^\prime_{+,\epsilon}(n^N)\right) \right)\right|^2 \right)^\frac{1}{2} \left(\int_{\Omega_T}\left| \nabla \Pi_N \phi \right|^2\right)^\frac{1}{2}.
        \end{aligned}
        \label{eq:semi-discr-ineq-dtn-1}
    \end{equation}

Using \eqref{eq:semi-discr-ineq3}, we obtain 
\begin{equation}
    \left|\int_{\Omega_T} \partial_t n^N \phi \right| \le C \left(\int_{\Omega_T}\left| \nabla \Pi_N \phi \right|^2\right)^\frac{1}{2} .
    \label{eq:semi-discr-ineq-dtn}
\end{equation}

Thus we can extract a subsequence such that
\begin{equation}
            \partial_t n^N \rightharpoonup  \partial_t n_{\sg,\epsilon}\text{ weakly in }  L^2(0,T;(H^1(\Omega))').
 \label{eq:semi-discr-ptn}
\end{equation}
        From \eqref{eq:semi-discr-n-weak} and \eqref{eq:semi-discr-ptn} and using the Lions-Aubin Lemma, we obtain the strong convergence
\begin{equation}
    n^N \to n_{\sg,\epsilon} \text{ strongly in }  L^2(0,T;L^2(\Omega)).
    \label{eq:semi-discr-conv-strong-n}
\end{equation}
Next, we need to prove the strong convergence of $n^N-\f\sg\gamma \vp^N$ in $L^2(0,T;H^1(\Omega))$. 
In order to do that we must bound the $L^2(0,T;(H^1(\Omega))')$ norm of its time derivative. 
Starting from the equation \eqref{eq:init-val2}, multiplying it by $-\f{\sg}{\gamma}$, adding $c^N_j$ and calculating its time derivative, we obtain
\[
    \f{d}{dt}\left(c^N_j-\f\sg\gamma d^N_j\right) = \f{d}{dt} c^N_j -\sg \lambda_j\f{d}{dt} \left(c^N_j-\f\sg\gamma d^N_j\right) - \f\sg\gamma \f{d}{dt}\int_\Omega \psi_-^\prime\left(n^N-\f\sg\gamma \vp^N \right)\phi_j.
\]
Multiplying the previous equation by $\phi_j \p_t\left(n^N - \f{\sg}{\gamma} \vp^N\right)$, summing over $j$ and integrating over $\Omega$, we obtain
\[
    \begin{aligned}
    \int_\Omega \left(\p_t\left(n^N-\f\sg\gamma \vp^N\right)\right)^2 &+ \sg \int_\Omega \big|\nabla \left(\p_t\left(n^N-\f\sg\gamma \vp^N \right) \right) \big|^2= \int_\Omega \p_t n^N \p_t\left(n^N-\f\sg\gamma \vp^N\right)  \\
    &  - \sum_j \int_\Omega \phi_j \p_t\left(n^N-\f\sg\gamma \vp^N \right)\f\sg\gamma \f{d}{dt}\int_\Omega \psi_-^\prime\left(n^N-\f\sg\gamma \vp^N \right)\phi_j \,\dd x\,\dd y.
    \end{aligned}
\]
Let us define $U^N=\p_t\left( n^N-\f\sg\gamma \vp^N\right)$ and rewrite the previous equation
\[
    \sg \int_\Omega |\nabla U^N|^2 +  \int_\Omega |U^N|^2 = \int_\Omega  \p_t n^N U^N  -  \f{\sg}{\gamma}\int_\Omega |U^N|^2 \psi^{\prime\prime}_-(n^N-\f{\sg}{\gamma}\vp^N) .
\]
From the Cauchy-Schwarz inequality, we obtain 
\begin{equation}
    0 \le ||\nabla U^N||_{L^2(\Omega)}^2 + \left(1-\f{\sg}{\gamma} ||\psi_-^{\prime\prime}||_\infty \right) ||U^N||_{L^2(\Omega)}^2 \le ||\p_t n^N||_{L^2(\Omega)} ||U^N||_{L^2(\Omega)}.
    \label{eq:boundUN}
\end{equation}
Finally, from the \eqref{eq:psi-minus} we obtain that
\[
    || U^N ||_{L^2\left(0,T;\left(H^1(\Omega)\right)'\right)} \le C.
\]
\begin{comment}
Starting from the second equation of the system \eqref{eq:CH-relax}, we can obtain an equation on the time derivative of $\vp$ (for details of this calculation we refer the reader to Proposition~\ref{prop:control-time-deriv}, here, in addition, we use the Galerkin discretization). Using this equation in its weak form and applying the previous Galerkin method, we obtain 
\[
    \sigma \int_\Omega \nabla U^N \nabla \phi + \int_\Omega U^N\phi = \int_\Omega \p_t n^N \phi - \f\sg\gamma \int_\Omega U^N \psi^{\prime\prime}_-(n^N-\f\sg\gamma\vp^N)\phi,
\]
where $U^N=\p_t\left( n^N-\f\sg\gamma \vp^N\right)$.
If we replace in the previous equation the test function by $U^N$, we obtain
\[
    \sg \int_\Omega |\nabla U^N|^2 +  \int_\Omega |U^N|^2 = \int_\Omega  \p_t n^N U^N  -  \f{\sg}{\gamma}\int_\Omega |U^N|^2 \psi^{\prime\prime}_-(n^N-\f{\sg}{\gamma}\vp^N) .
\]
From the Cauchy-Schwarz inequality, we obtain 
\[
    0 \le ||\nabla U^N||_{L^2(\Omega)}^2 + \left(1-\f{\sg}{\gamma} ||\psi_-^{\prime\prime}||_\infty \right) ||U^N||_{L^2(\Omega)}^2 \le ||\p_t n^N||_{L^2(\Omega)} ||U^N||_{L^2(\Omega)}.
\]
Taking into account the assumptions \eqref{eq:psi-minus} and \eqref{eq:semi-discr-ineq-dtn}, the following inequality holds
\[
    ||U^N||_{L^2(0,T;(H^1(\Omega))')} \le C.
\]
\textcolor{red}{Sure of the space here?}
\end{comment}
Therefore, we can extract a subsequence such that 
\begin{equation}
    \p_t\left(n^N-\f \sg\gamma \vp^N \right) \rightharpoonup \p_t\left(n_{\sg,\epsilon}-\f \sg\gamma \vp_{\sg,\epsilon} \right)  \text{ weakly in }  L^2(0,T;(H^1(\Omega))').
    \label{eq:semi-discr-conv-dt-n-vp}
\end{equation}
Using \eqref{eq:semi-discr-conv-dt-n-vp} and \eqref{eq:H^2-n-vp} and the Lions-Aubin lemma we obtain the following strong convergence
\begin{equation}
    n^N-\f \sg\gamma \vp^N \to n_{\sg,\epsilon}-\f \sg\gamma \vp_{\sg,\epsilon} \text{ strongly in }  L^2(0,T;H^1(\Omega)).
    \label{eq:semi-discr-conv-strong-n-vp}
\end{equation}
 \\[5pt]  
\textit{Step 3. Limiting equation.} 
The main difficulty to pass to the limit in the equation \eqref{eq:semi-discr3} relies mainly on the convergence of the term $\int_\Omega \psi_-^\prime(n^N-\f\sg\gamma \vp^N) \phi_j$ which is solved using the strong convergence \eqref{eq:semi-discr-conv-strong-n-vp} and the properties \eqref{eq:psi-minus}.  Therefore, we obtain
\begin{equation}
    \psi_-^\prime(n^N-\f \sg\gamma \vp^N) \to \psi_-^\prime(n_{\sg,\epsilon}-\f \sg\gamma \vp_{\sg,\epsilon})\quad  \text{ a.e. in } \Omega_T. 
    \label{eq:conv_ponct_psim} 
\end{equation}
Then combining the convergences \eqref{eq:semi-discr-conv-n-vp}, \eqref{eq:semi-discr-conv-vp}, \eqref{eq:conv_ponct_psim} and the Lebesgue dominated convergence theorem, we pass to the limit in the equation \eqref{eq:semi-discr3}.
We can also pass to the limit in the first equation \eqref{eq:semi-discr2} by the standard manner (see \cite{lions_quelques_1969}), using the strong convergence \eqref{eq:semi-discr-conv-strong-n}, the properties of the mobility \eqref{eq:reg-mob-continuity} and the potential \eqref{eq:assume-psi}.
Altogether, we obtain the limiting system \eqref{eq:limit-reg-rel-sys}.
\end{proof}

%-----------------------------------------------------------------
\subsection{Energy, entropy and a priori estimates}
\label{sec:estimates}
%-----------------------------------------------------------------
%
The relaxed and regularized system~\eqref{eq:reg_pb} comes with an energy and an entropy. These provide us with estimates which are useful to prove the existence of global weak solutions of~\eqref{eq:reg_pb} and their convergence to the weak solutions of the original DHC equation or to the  RDHC as $\epsilon$ and/or $ \sg \rightarrow 0$.

Being given a solution $(n_{\sg,\epsilon},\vp_{\sg,\epsilon})$ satisfying Theorem~\ref{th:existence-regularized}, we define the energy associated with the regularized potential $\psi_{+,\epsilon}$ and relaxed system as
\begin{equation}
\cae_{\sigma,\epsilon} [n_{\sg,\epsilon}] = \int_\Omega \left[ \psi_{+,\epsilon}(n_{\sg,\epsilon})  +  \f{\gamma}{2}|\nabla (n_{\sg,\epsilon}-\f{\sigma}{\gamma}\vp_{\sg,\epsilon})|^2 + \f{\sg}{2 \gamma} |\vp_{\sg,\epsilon}|^2 + \psi_-(n_{\sg,\epsilon} -\f{\sg}{\gamma}\vp_{\sg,\epsilon}) \right] ,
\label{eq:energy}
\end{equation}
where $\vp_{\sg,\epsilon}$ is obtained from $n_{\sg,\epsilon}$ by solving the elliptic equation in~\eqref{eq:reg_pb}. 
Notice that $\cae_{\sigma,\epsilon}[n_{\sg,\epsilon}]$ is lower bounded, uniformly in $\epsilon$ and $\sigma$, thanks to the  assumptions on  $\psi_-$ in~\eqref{eq:psi-minus} and  the construction of $\psi_{\epsilon,+}$ in~\eqref{eq:assume-psi}.

\begin{proposition}[Energy]
Consider a solution $(n_{\sg,\epsilon}, \vp_{\sg,\epsilon})$ of~\eqref{eq:reg_pb}--\eqref{eq:bound_reg} defined by Theorem~\ref{th:existence-regularized}, then, the energy of the system $\cae_{\sigma,\epsilon} $ satisfies
\begin{equation}
\f{d}{dt} \cae_{\sigma,\epsilon} [n_{\sg,\epsilon}(t)]  = -\int_\Omega B_\epsilon(n_{\sg,\epsilon}) \big|\nabla (\vp_{\sg,\epsilon} + \psi'_{+,\epsilon}(n_{\sg,\epsilon}))\big|^2 \leq 0 .
\label{eq:deriv-energy}
\end{equation}
\label{prop:energy}
\end{proposition}

As a consequence, we obtain a first a priori estimate 
\begin{equation}
\cae_{\sigma,\epsilon} [n_{\sg,\epsilon}(T)] +\int_0^T  \int_\Omega B_\epsilon(n_{\sg,\epsilon}) \big|\nabla (\vp_{\sg,\epsilon} + \psi'_{+,\epsilon}(n_{\sg,\epsilon}))\big|^2 = \cae_{\sigma,\epsilon} [n^0].
\label{eq:energy-0}
\end{equation}

\begin{proof}
To establish the energy of the regularized system, we begin with multiplying the first equation of~\eqref{eq:reg_pb} by $\vp_{\sg,\epsilon} + \psi_{+,\epsilon}^\prime(n_{\sg,\epsilon})$. Then, we integrate on the domain  $\Omega$  and use the second boundary condition~\eqref{eq:bound_reg} to obtain
\[
\int_\Omega  [ \vp_{\sg,\epsilon}  +  \psi'_{+,\epsilon}(n_{\sg,\epsilon}) ]\p_t n_{\sg,\epsilon}  = - \int_\Omega B_\epsilon(n_{\sg,\epsilon})  |\nabla ( \vp_{\sg,\epsilon} + \psi_{+,\epsilon}' (n_{\sg,\epsilon}))|^2 .
\]
Since $ \psi'_{+,\epsilon}(n_{\sg,\epsilon}) \p_t n_{\sg,\epsilon}  =  \p_t \psi_{+,\epsilon}(n_{\sg,\epsilon}) $, to retrieve the energy equality~\eqref{eq:deriv-energy} we need to focus on the calculation of $\int_\Omega \vp_{\sg,\epsilon} \p_t n_{\sg,\epsilon}$.  We write 
\[
\int_\Omega  \vp_{\sg,\epsilon}\p_t n_{\sg,\epsilon}  = \int_\Omega  \vp_{\sg,\epsilon}\p_t [n_{\sg,\epsilon} -\f{\sigma}{\gamma}\vp_{\sg,\epsilon}]
+  \frac{d}{dt}\int_\Omega   \f{\sg}{2 \gamma} |\vp_{\sg,\epsilon}|^2,
\]
and using the second equation of \eqref{eq:reg_pb}, we rewrite the first term as 
\[ \begin{aligned}
\int_\Omega \vp_{\sg,\epsilon}\p_t [n_{\sg,\epsilon} -\f{\sigma}{\gamma}\vp_{\sg,\epsilon}] 
&=\int_\Omega [- \gamma \Delta (n_{\sg,\epsilon} -\f{\sigma}{\gamma}\vp_{\sg,\epsilon}) + \psi_-^\prime(n_{\sg,\epsilon} - \f \sg \gamma \vp_{\sg,\epsilon}) ] \p_t [n_{\sg,\epsilon} -\f{\sigma}{\gamma}\vp_{\sg,\epsilon}] 
\\
& = \f{d}{dt} \int_\Omega \f \gamma 2 \nabla (n_{\sg,\epsilon}-\f{\sigma}{\gamma}\vp_{\sg,\epsilon})|^2 + \psi_-(n_{\sg,\epsilon}-\f{\sg}{\gamma}\vp_{\sg,\epsilon}),
\end{aligned} \]
where we have used the first boundary condition~\eqref{eq:bound_reg}.

Altogether, we have recovered the expression~\eqref{eq:energy} and the equality~\eqref{eq:deriv-energy}.
\end{proof}

\medskip

We can now turn to the entropy inequality. It is classical to define the mapping $\phi_\epsilon: [0,\infty)\mapsto [0,\infty)$ 
\begin{equation}
\phi^{\prime \prime}_\epsilon(n) = \f{1}{B_\epsilon(n)}, \qquad  \phi_\epsilon(0) = \phi_\epsilon'(0)=0,
\label{eq:assumption-entropy}
\end{equation}
 which is  well defined because $B_\epsilon \in C(\mathbb{R}, \mathbb{R}^+)$ from~\eqref{eq:bn-assumption}.  For a nonnegative function $n(x)$, we define the entropy as 
\[
\Phi_{\epsilon}[n] = \int_\Omega \phi_\epsilon\big(n(x)\big)dx.
\]

\begin{proposition}[Entropy] 
Consider a solution of~\eqref{eq:reg_pb}--\eqref{eq:bound_reg} defined by Theorem~\ref{th:existence-regularized}, then the entropy of the system satisfies
\begin{equation}
\begin{aligned}
\f{ d \Phi_{\epsilon}[n_{\sg,\epsilon}(t)]}{dt} = -\int_\Omega \gamma \left| \Delta\left( n_{\sg,\epsilon} - \f{\sg}{\gamma} \vp_{\sg,\epsilon}\right)\right|^2 + \f{\sg}{\gamma}|\nabla \vp_{\sg,\epsilon}|^2 
&+ \psi''_-(n_{\sg,\epsilon} - \f{\sg}{\gamma} \vp_{\sg,\epsilon}) \left| \nabla(n_{\sg,\epsilon} - \f{\sg}{\gamma}\vp_{\sg,\epsilon}) \right|^2 \\
&+ \psi^{\prime\prime}_{+,\epsilon}(n_{\sg,\epsilon})|\nabla n_{\sg,\epsilon}|^2.
\end{aligned}
\label{eq:entropy}
\end{equation}
\end{proposition}
%--------------------------
Notice that the dissipation terms are all well defined by our definition of solution in Theorem~\ref{th:existence-regularized}. However, 
the equality \eqref{eq:entropy} does not provide us with a direct a priori estimate because of the negative term $\psi^{\prime \prime}_-$, therefore we have to combine it with the energy identity to write  
\[
\begin{aligned}
\Phi_{\epsilon}[n_{\sg,\epsilon}(T)] + \int_{\Omega_T} & \left[ \gamma \left| \Delta\left( n_{\sg,\epsilon} - \f{\sg}{\gamma} \vp_{\sg,\epsilon}\right)\right|^2 + \f{\sg}{\gamma}|\nabla \vp_{\sg,\epsilon}|^2 + \psi^{\prime\prime}_{+,\epsilon}(n_{\sg,\epsilon})|\nabla n_{\sg,\epsilon}|^2\right] 
\\
&\leq  \Phi_{\epsilon}[n^0] + \frac{2T}{\gamma} \| \psi''_- \|_\infty \;  \cae_{\sg,\epsilon}[n^0].
\end{aligned}
\]

\begin{proof}
We compute, using the definition of $\phi^{\prime \prime}_\epsilon$, 
\begin{equation}
\begin{aligned}
 \int_\Omega \p_t \phi_\epsilon(n_{\sg,\epsilon})
 &= \int_\Omega \p_t n_{\sg,\epsilon} \phi^\prime_\epsilon(n_{\sg,\epsilon}) \\
 &= \int_\Omega  \nabla \cdot \left[ B_\epsilon(n_{\sg,\epsilon}) \nabla(\vp_{\sg,\epsilon} + \psi_{+,\epsilon}^\prime(n_{\sg,\epsilon}) ) \right] \phi^\prime_\epsilon(n_{\sg,\epsilon}) \\
&= - \int_\Omega B_\epsilon(n_{\sg,\epsilon}) \nabla(\vp_{\sg,\epsilon} + \psi_{+,\epsilon}^\prime(n_{\sg,\epsilon}) )  \phi^{\prime\prime}_\epsilon(n_{\sg,\epsilon}) \nabla n_{\sg,\epsilon} \\
&= - \int_\Omega  \nabla(\vp_{\sg,\epsilon} + \psi_{+,\epsilon}^\prime(n_{\sg,\epsilon}) ) \nabla n_{\sg,\epsilon}
 \\
&= - \int_\Omega  \nabla \vp_{\sg,\epsilon} \nabla( n_{\sg,\epsilon} - \f \sg \gamma  \vp_{\sg,\epsilon} ) + \psi_{+,\epsilon}^{\prime \prime}(n_{\sg,\epsilon})  |\nabla n_{\sg,\epsilon}|^2 +\f \sg \gamma  | \nabla \vp_{\sg,\epsilon}|^2 .  
\end{aligned}
\label{eq:entropy-reg-debut}
\end{equation} 
To rewrite the term $\int_\Omega  \nabla \vp_{\sg,\epsilon} \nabla( n_{\sg,\epsilon} - \f \sg \gamma  \vp_{\sg,\epsilon} ) $, we use the second equation of the regularized system~\eqref{eq:reg_pb}
\begin{equation}
\vp_{\sg,\epsilon} = -\gamma \Delta\left(n_{\sg,\epsilon} - \f{\sg}{\gamma}\vp_{\sg,\epsilon}\right) + \psi_-^\prime(n_{\sg,\epsilon}-\f{\sg}{\gamma}\vp_{\sg,\epsilon}).
\label{eq:rewrite-2-reg}
\end{equation}
Using  \eqref{eq:rewrite-2-reg} and the boundary condition~\eqref{eq:bound_reg}, we can  rewrite the term under consideration as 
\[
\begin{aligned}
\int_\Omega \vp_{\sg,\epsilon} \Delta\left(n_{\sg,\epsilon} - \f{\sg}{\gamma}\vp_{\sg,\epsilon} \right) &=  \int_\Omega -\gamma \left| \Delta\left( n_{\sg,\epsilon} - \f{\sg}{\gamma} \vp_{\sg,\epsilon}\right)\right|^2 + \psi^{\prime}_-(n_{\sg,\epsilon} - \f{\sg}{\gamma} \vp_{\sg,\epsilon})\Delta\left(n_{\sg,\epsilon} - \f{\sg}{\gamma}\vp_{\sg,\epsilon} \right)
\\
&= - \int_\Omega \gamma \left| \Delta\left( n_{\sg,\epsilon} - \f{\sg}{\gamma} \vp_{\sg,\epsilon}\right)\right|^2 
+  \psi''_-(n_{\sg,\epsilon} - \f{\sg}{\gamma}\vp_{\sg,\epsilon})  \left| \nabla (n_{\sg,\epsilon} - \f{\sg}{\gamma}\vp_{\sg,\epsilon} ) \right|^2.
\end{aligned} \]
Injecting this equality into \eqref{eq:entropy-reg-debut}, we obtain the identity~\eqref{eq:entropy}.
\end{proof}

%------------------------------------------------
\subsection{Inequalities}
\label{sec:inequalities}
%------------------------------------------------
From the energy and entropy properties, we can conclude the following a priori bounds, where we assume that the initial data has finite energy and entropy, 
\begin{equation}
    \f{\sg}{2\gamma}  \int_\Omega \left|\vp_{\sg,\epsilon}(t)\right|^2 \leq \cae_{\sigma,\epsilon} [n^0], \qquad \forall  t \geq 0,
    \label{eq:vp-est}
 \end{equation}
 \begin{equation}
 \f{\sg}{\gamma} \int^T_0 \int_\Omega  |\nabla \vp_{\sg,\epsilon}|^2 
 \leq \Phi_{\epsilon}[n^0] + \frac{2T}{\gamma} \| \psi^{\prime \prime}_-\|_{\infty} \cae_{\sg,\epsilon}[n^0], \qquad \forall T\geq 0 ,
 \label{eq:grad-vp-est}
 \end{equation}
 \begin{equation}
    \f{\gamma}{2} \int_\Omega \left|\nabla \big(n_{\sg,\epsilon}(t)-\f{\sigma}{\gamma}\vp_{\sg,\epsilon}(t)\big) \right|^2 \leq \cae_{\sigma,\epsilon} [n^0], \qquad \forall  t \geq 0,
    \label{eq:grad-est}
 \end{equation}
 \begin{equation}
 \int_0^T \int_\Omega \left| \Delta\big( n_{\sg,\epsilon} - \f{\sg}{\gamma} \vp_{\sg,\epsilon}\big)\right|^2 \leq \Phi_{\epsilon}[n^0] + \frac{2T}{\gamma}\| \psi^{\prime \prime}_-\|_{\infty}  \cae_{\sg,\epsilon}[n^0], \qquad \forall T\geq 0,
 \label{eq:lap-est}
 \end{equation}
 \begin{equation}
 \int_0^T  \int_\Omega B_\epsilon(n_{\sg,\epsilon}) \big|\nabla (\vp_{\sg,\epsilon} + \psi'_{+,\epsilon}(n_{\sg,\epsilon}))\big|^2 \le \cae_{\sigma,\epsilon} [n^0], \qquad \forall T\geq 0.
 \label{eq:flux-est}
 \end{equation}

 \begin{proposition}[Compactness of time derivatives]
    Consider a solution $(n_{\sg,\epsilon}, \vp_{\sg,\epsilon})$ of~\eqref{eq:reg_pb}--\eqref{eq:bound_reg} defined by Theorem~\ref{th:existence-regularized}, then, the following inequalities hold for $\sigma$ small enough
    \begin{align}
    ||\p_t n_{\sg,\epsilon}||_{L^2(0,T;(H^1(\Omega))^\prime)} &\le C, \label{eq:control-deriv-n}\\
    ||\p_t \left( n_{\sg,\epsilon} -\f\sg\gamma \vp_{\sg,\epsilon}\right)||_{L^2(0,T;(H^1(\Omega))^\prime)} &\le \AP{C(\sigma)}.\label{eq:control-deriv-n-vp}
    \end{align}
    \label{prop:control-time-deriv}
\end{proposition}
\begin{proof}
    For any test function $\chi\in L^2(0,T;H^1(\Omega))$ we obtained from \eqref{eq:flux-est}
    \[
        \begin{aligned}
            \left|\int_{\Omega_T} \p_t n_{\sg,\epsilon} \chi  \right| &= \left| \int_{\Omega_T}B_\epsilon(n_{\sg,\epsilon}) \nabla\left(\vp_{\sg,\epsilon}+\psi^\prime_{+,\epsilon}(n_{\sg,\epsilon}) \right) \nabla \chi\right|\\
            &\le \left( \int_{\Omega_T} \left|B_\epsilon(n_{\sg,\epsilon}) \nabla\left(\vp_{\sg,\epsilon} +\psi^\prime_{+,\epsilon}(n_{\sg,\epsilon}) \right) \right|^2 \right)^{1/2} || \nabla \chi ||_{L^2(\Omega_T)},\\
            &\le C || \nabla \chi ||_{L^2(\Omega_T)}.
        \end{aligned}  
    \]
This proves \eqref{eq:control-deriv-n}. \par
To prove \eqref{eq:control-deriv-n-vp}, we compute the time derivative of equation for $\vp_{\sg,\epsilon} $ in the distribution sense
\[
    \sg \int_{\Omega_T} \nabla U_{\sg,\epsilon} \nabla \chi + \int_{\Omega_T} U_{\sg,\epsilon} \chi = \int_{\Omega_T} \p_t n_{\sg,\epsilon}\chi -\f{\sg}{\gamma}\int_{\Omega_T} U_{\sg,\epsilon} \psi^{\prime\prime}_-(n_{\sg,\epsilon}-\f\sg\gamma \vp_{\sg,\epsilon})\chi,
\]
where $U_{\sg,\epsilon} = \p_t\left(n_{\sg,\epsilon}-\f\sg\gamma\vp_{\sg,\epsilon} \right)$ and we have used the fact that $\left(n_{\sg,\epsilon}-\f\sg\gamma\vp_{\sg,\epsilon} \right)$, $n_{\sg,\epsilon}$ and $\psi^{\prime}_-(n_{\sg,\epsilon}-\f\sg\gamma \vp_{\sg,\epsilon})$ are smooth.
Then, we can choose $\chi = U_{\sg,\epsilon}$, to obtain
\[
    \sg \int_{\Omega_T} |\nabla U_{\sg,\epsilon}|^2 + \int_{\Omega_T} |U_{\sg,\epsilon}|^2 = \int_{\Omega_T} \p_t n_{\sg,\epsilon} U_{\sg,\epsilon} -\f{\sg}{\gamma}\int_{\Omega_T} |U_{\sg,\epsilon}|^2 \psi^{\prime\prime}_-(n_{\sg,\epsilon}-\f\sg\gamma \vp_{\sg,\epsilon}).
\]
Using the fact that $\f\sg\gamma ||\psi_-^{\prime\prime}||_\infty<1$ from \eqref{eq:psi-minus}, the Cauchy-Schwarz inequality gives 
\[
   \AP{\sigma} ||\nabla U_{\sg,\epsilon}||^2_{L^2(\Omega_T)} + \alpha ||U_{\sg,\epsilon}||^2_{L^2(\Omega_T)} \le ||\p_t n_{\sg,\epsilon}||_{L^2(\Omega_T)} ||U_{\sg,\epsilon}||_{L^2(\Omega_T)},
\]
where $\alpha = 1-\f\sg\gamma \|\psi_-^{\prime\prime}\|_\infty>0$. Altogether, we obtain the bound~\eqref{eq:control-deriv-n-vp}.
\end{proof}

%---------------------------------------------------------------
\section{Existence: convergence as $\epsilon \to 0$}
\label{sec:existence}
%---------------------------------------------------------------

The next step is to prove the existence of global weak solutions for the RDCH system~\eqref{eq:CH-relax} by letting $\e$ vanish.  This means that for all test functions $\chi \in L^2(0,T;H^1(\Omega))\cap L^\infty(\Omega_T)$ with  $\nabla \chi  \cdot \nu =0$ on $\p \Omega \times (0,T)$, it holds
\[
        \begin{aligned}
            \int_0^T <\chi,\p_t n_\sigma> &= \int_{\Omega_T} b(n_\sigma) \nabla\left(\vp_\sigma + \psi_+^\prime(n_\sigma)\right)\nabla \chi ,\\
            \sigma \int_{\Omega_T} \nabla \vp_\sigma \nabla \chi + \int_{\Omega_T} \vp_\sigma \chi & = \gamma \int_{\Omega_T} \nabla n_\sigma \nabla \chi + \int_{\Omega_T} \psi_-^\prime(n_\sigma-\frac{\sigma}{\gamma} \vp_\sigma) \chi.
        \end{aligned}
\]
We establish the following
\begin{theorem} [Existence for $\sg >0, \, \e=0$] \label{th:existence}
Assume an initial condition satisfying $0\leq n^0 \leq 1$, with finite energy and entropy. Then, for $\sg$ small enough,  there exists a global weak solution  $(n_\sigma,\vp_\sigma)$ of the RDCH  equation~\eqref{eq:CH-relax}, \eqref{eq:CH-bound-relax} such that
    \begin{align}
        n_\sigma& \in L^2(0,T;H^1(\Omega)), \qquad  \p_t n_\sigma  \in  L^2\big(0,T;(H^1(\Omega))^\prime \big). 
        \\
        \vp_\sigma& \in L^2(0,T;H^1(\Omega)),\\
        n_\sigma - \f\sg\gamma \vp_\sg &\in L^2(0,T;H^2(\Omega)), \qquad  \p_t \left(n_\sigma-\f\sg\gamma \vp_\sg \right)  \in  L^2\big(0,T;(H^1(\Omega))^\prime \big). 
    \end{align}
    \begin{equation}
        0\le n_\sigma \leq  1, \qquad \text{ a.e. in }  \Omega_T,
        \label{eq:threshold}
    \end{equation}
 and $n_\sg <1$ a.e. if $b$ vanishes fast enough at $1$ so that $\phi(1)=\infty$ (see \eqref{eq:assumption-entropy_L}).
\end{theorem}
%------------------------------------------------
\begin{proof}
    The proof relies on compactness results and the inequalities presented in section~\ref{sec:inequalities}. From these inequalities, we can extract subsequences of $(n_{\sg,\epsilon},\vp_{\sg,\epsilon})$ such that the following convergences for $\epsilon \to 0$ hold for all $T>0$. \par
\noindent \textit{ Step 1. Weak limits.}
    From \eqref{eq:vp-est} and  \eqref{eq:grad-vp-est}, we immediately have 
    \begin{equation}
     \vp_{\sg,\epsilon}  \rightharpoonup \vp_\sigma \text{ in } L^2 \big((0,T); H^1(\Omega ) \big).
    \label{eq:conv_svp-exist}
    \end{equation} 
    Next, from \eqref{eq:grad-est}, and the above convergence, we conclude
    \begin{equation}
    n_{\sg,\epsilon}  \rightharpoonup n_\sigma \text{ weakly in } L^2\big(0,T; H^1(\Omega) \big), 
    \label{eq:conv-weak-n-vp-exist}
    \end{equation}
    Finally from \eqref{eq:control-deriv-n} and \eqref{eq:control-deriv-n-vp}, we have
\begin{align*}
\p_t n_{\sg,\epsilon}  &\rightharpoonup \p_t n_\sigma  \text{ weakly in } L^2\big(0,T;(H^1(\Omega))^\prime \big),\\
\p_t \left(n_{\sg,\epsilon} -\f\sg\gamma \vp_{\sg,\epsilon} \right)&\rightharpoonup \p_t \left(n_\sigma -\f\sg\gamma \vp_\sg \right)\text{ weakly in } L^2\big(0,T;(H^1(\Omega))^\prime \big).
\end{align*}

\noindent{\em Step 2. Strong convergence.}
Therefore, from the Lions-Aubin lemma and Proposition~\ref{prop:control-time-deriv} we obtain the strong convergences
\begin{equation}
    n_{\sigma,\epsilon} \to n_\sigma \in L^2(0,T;L^2(\Omega)).
    \label{eq:conv-n-exist}
\end{equation} 
\begin{equation}
    n_{\sigma,\epsilon} - \f\sg\gamma \vp_{\sg,\epsilon} \to n_\sigma- \f\sg\gamma \vp_{\sg} \in L^2(0,T;H^1(\Omega)).
    \label{eq:conv-n-vp-exist}
\end{equation} 

\noindent \textit{Step 3. Bounds $0\leq n_\sigma  \le 1$.}
To prove these bounds on $n_\sg$, several authors have used the entropy relation. In the context of DCH equation with double-well potentials featuring singularities at $n=1$ and $n=-1$, the solution lies a.e.  in the interval $-1<n<1$. Elliott and Garcke~\cite{elliott_cahn-hilliard_1996} prove this result using the definition of the regularized entropy and by a contradiction argument. For single-well potential, Agosti~\textit{et~al.}~\cite{agosti_cahn-hilliard-type_2017} used a reasoning on the measure of the set of solutions outside the set $0\le n <1$ and find contradictions with the boundedness of the entropy. This is the route we follow here. In the following, all functions are defined almost everywhere.
\par
We begin by the upper bound . For $\alpha >0$, we consider the set 
\[
    V^\e_\alpha = \{(t,x) \in \Omega_T | n_{\sigma,\epsilon}(t,x) \ge 1+\alpha \}.
\]
For $A>0$, there exists a small $\epsilon_0$ such that the following estimate holds for every $\epsilon \le \epsilon_0$  
%Consider $A>0$ (large) and for $\e$ small enough,  such that (because $b(1)=0$)
\[
    \phi_{\epsilon}^{\prime\prime}(n) = \frac{1}{b(1-\epsilon)}\geq 2 A \qquad \forall n \geq 1, \; \forall \epsilon > 0.    
\]
Thus, integrating this quantity twice, we obtain
\[
    \phi_{\epsilon}(n) \geq A (n-1)^2  \qquad \forall n \geq 1.  
\]
Also, from \eqref{eq:entropy}, we know that the entropy is uniformly bounded in $\e$. Therefore, we obtain
\[
| V^\e_\alpha | A \alpha^2 \leq \int_{\Omega_T}  \phi_{\epsilon}(n_{\sigma,\epsilon}(t,x)) \leq C(T), \qquad | V^\e_\alpha | \leq \frac{C(T)}{ A \alpha^2}.
\]
In the limit $\e \to 0$, using Fatou's lemma and the strong convergence of $n_{\sg,\epsilon}$, we conclude that
\[
\big| \{(t,x) \in \Omega_T | n_{\sigma}(t,x) \ge 1+\alpha \} \big| \leq \frac{C(T)}{ A \alpha^2}, \qquad \forall A>0.
\]
In other words $n_{\sigma}(t,x) \le 1+\alpha$ for all $\alpha>0$, which means $n_{\sigma}(t,x) \le 1$.

The same argument also gives $n_\sg \geq 0$ and we do not repeat it.\par

The second statement, $n_\sg <1$ under the assumption $\phi(1)=+\infty$, is a consequence of the bound
\[
\int_{\Omega_T}  \phi (n_{\sigma}(t,x)) \leq C(T),
\]
which holds true by strong convergence of $n_{\sigma,\epsilon}$ and because $\phi_\e \nearrow \phi$ as $\e \searrow 0$.\par
%------------------------------
\noindent \textit{Step 4. Limiting equation.}
Finally, it remains to show that the limit of subsequences satisfies the RDCH equation in the weak form. 
Firstly, using the weak convergences ~\eqref{eq:conv_svp-exist}--\eqref{eq:conv-weak-n-vp-exist}, the strong convergence \eqref{eq:conv-n-vp-exist} and the properties of $\psi_-^\prime$ gathered from \eqref{eq:psi-minus}, we can pass to the limit in the standard way to obtain the second equation of the limit system.\par
\noindent To conclude the proof, we need to prove the following weak convergence, recalling that~\eqref{eq:flux-est} provides a uniform $L^2$ bound over $\Omega_T$, on $J_{\sigma,\epsilon}$
\begin{equation}
    J_{\sigma,\epsilon} := -B_\epsilon(n_{\sigma,\epsilon})\nabla(\vp_{\sigma,\epsilon} + \psi^\prime_{+,\epsilon}(n_{\sigma,\epsilon})) \rightharpoonup -b(n_\sigma) \nabla(\vp_\sigma + \psi_+^\prime(n_\sigma)) \text{ weakly in } L^2(\Omega_T). 
\end{equation}
The convergence of $B_\epsilon(n_{\sigma,\epsilon})\nabla\vp_{\sigma,\epsilon} $ follows from the weak convergence in $L^2(\Omega_T)$ of $\nabla\vp_{\sigma,\epsilon}$ and the strong convergence $B_\epsilon(n_{\sigma,\epsilon}) \rightarrow b(n_{\sigma})$ in all $L^p(\Omega_T)$, $1\leq p <\infty$  which follows from~\eqref{eq:conv-n-exist} and the fact that $B_\epsilon(.) \rightarrow b(.)$ uniformly. 

Because of the singularity  $\psi'_+(1)=\infty$, we use the assumption~\eqref{eq:assbpsi} 
and that $B_\epsilon(\cdot) \psi^{\prime\prime}_{+,\epsilon}(\cdot) \rightarrow b(\cdot)\psi^{\prime\prime}_{+}(\cdot)$ uniformly and thus
 $B_\epsilon(n_{\sg,\e}) \psi^{\prime\prime}_{+,\epsilon}(n_{\sg,\e}) \rightarrow b(n_{\sg,\e})\psi^{\prime\prime}_{+}(n_{\sg,\e})$ a.e. in $\Omega_T$ 
 This achieve the proof.
\end{proof}

\medskip

It is easy to check that the energy and entropy relations~\eqref{eq:energy_L}, \eqref{eq:entropy_L} hold, at least as inequalities. In the sequel we only use the a priori bounds coming from the limiting procedure.

%------------------------------------------------------
\section{Convergence as $\sigma \to 0$}
\label{sec:cvg}

We are now ready to study the limit of the relaxed solution $n_\sg$ towards a solution of the DCH equation,
Our main result is as follows.
%----------------------------------------
\begin{theorem}[Limit $\sg =0$]
Let $(n_{\sg,\epsilon}, \vp_{\sg,\epsilon})$ be a sequence of weak solutions of the RDHC system~\eqref{eq:reg_pb} with initial conditions $n^0$, $0\leq n^0 <1$,  with finite energy  and entropy. Then, as $\e, \; \sg \to 0$,  we can extract a subsequence of $(n_{\sg,\epsilon}, \vp_{\sg,\epsilon})$ such that
\begin{equation}
\vp_{\sg,\epsilon} \rightharpoonup - \gamma \Delta n + \psi'_-( n)  \quad \text{ weakly in } \; L^2(\Omega_T),
\label{eq:weak-phi-conv}
\end{equation}
\begin{equation}
    n_{\sg,\gamma}-\f\sg\gamma\vp_{\sg,\epsilon} \rightarrow n  \quad \text{ strongly in } \; L^2(0,T;H^1(\Omega)),
    \label{eq:strong-n-vp-conv}
\end{equation}
\begin{equation}
n_{\sg,\epsilon}, \;  \nabla n_{\sg,\epsilon} \rightarrow n, \; \nabla n \quad \text{ strongly in } \; L^2(\Omega_T), \text{ and }\; 0 \leq n  \leq1, 
\label{eq:strong-n-conv}
\end{equation}
and $n_\sg <1$ a.e. if $b$ vanishes fast enough at $1$ so that $\phi(1)=\infty$.
\begin{equation}
\p_t n_{\sg,\epsilon}  \rightharpoonup \p_t n  \text{ weakly in } L^2\big(0,T;(H^1(\Omega))^\prime \big).
\label{cv:time derivative}
\end{equation}
This limit $n$  satisfies the DCH system \eqref{eq:CH} in the  weak sense.
\label{th:sgto0}
\end{theorem}

We recall the definition of weak solutions; for all $\chi \in L^2(0,T;H^2(\Omega))\cap L^\infty(\Omega_T)$ with  $\nabla \chi  \cdot \nu =0$ on $\p \Omega \times (0,T)$, 
\begin{equation}
\begin{cases}
\int_0^T <\chi, \p_t n> &= \int_{\Omega_T} J \cdot \nabla \chi, \\
\int_{\Omega_T} J \cdot \nabla \chi &= - \int_{\Omega_T} \gamma   \Delta n  \left[ b'(n)\nabla n \cdot\nabla \chi + b(n) \Delta \chi \right] + (b\psi^{\prime \prime})(n) \nabla n \cdot \nabla \chi.
\end{cases}
\label{eq:weak-system}
\end{equation}

\begin{proof} We gathered, from the energy and entropy estimates of section~\ref{sec:estimates}, the a priori bounds of the section~\ref{sec:inequalities}. \par
\vspace{5pt}
\textit{Step 1. Weak limits.}
From the above mentioned inequalities, we can extract subsequences of $(n_{\sg,\epsilon},\vp_{\sg,\epsilon})$ such that the following convergences hold for all $T>0$. From \eqref{eq:vp-est} and  \eqref{eq:grad-vp-est}, we immediately have 
\begin{equation}
\sg \vp_{\sg,\epsilon} \to 0  \text{ in } L^2 \big(0,T; H^1(\Omega ) \big).
\label{eq:conv_svp}
\end{equation} 
Next, from \eqref{eq:grad-est}, and the above convergence, we conclude
\[
n_{\sg,\epsilon}  \rightharpoonup n \text{ weakly in } L^2\big(0,T; H^1(\Omega) \big), 
\]
and \eqref{eq:lap-est} gives directly 
\begin{equation}
\Delta (n_{\sg,\epsilon}-\f{\sg}{\gamma}\vp_{\sg,\epsilon}  )\rightharpoonup \Delta n  \text{ weakly in } L^2(\Omega_T).
\label{eq:conv_lapn}
\end{equation}
This latter convergence is obtained in the distribution sense using integration per parts, for all test function $\chi \in \mathcal{D}(\Omega_T)$
\[
    \int_{\Omega_T}  \Delta \big(n_{\sg,\epsilon}-\f{\sg}{\gamma}\vp_{\sg,\epsilon}  \big) \chi = - \int_{\Omega_T}  \nabla \big(n_{\sg,\epsilon}-\f{\sg}{\gamma}\vp_{\sg,\epsilon}  \big) \nabla \chi.
\] 
Then using \eqref{eq:conv_svp}, we obtain \eqref{eq:conv_lapn}.
The system of equations can also be used to complement these results. We find
\[
\vp_{\sg,\epsilon} \rightharpoonup \vp \text{ weakly in } L^2(\Omega_T),
\]
 using the second equation of the system \eqref{eq:reg_pb} and triangular inequality,
\[
\|\vp_{\sg,\epsilon} \|_{L^2(\Omega_T)} \le \gamma \|\Delta(n_{\sg,\epsilon}-\f{\sg}{\gamma} \vp_{\sg,\epsilon})\|_{L^2(\Omega_T)} + \| \psi_-^\prime(n_{\sg,\epsilon}-\f{\sg}{\gamma} \vp_{\sg,\epsilon}) \|_{L^2(\Omega_T)} .
\]
Finally from \eqref{eq:flux-est} and the equation on $n_{\sg,\epsilon}$ itself, we conclude~\eqref{cv:time derivative}.
\\[5pt]
\textit{Step 2. Strong convergence.} We continue with proving the strong convergences in~\eqref{eq:strong-n-conv}. From the inequality~\eqref{eq:lap-est}, we know that $\Delta \left( n_{\sg,\epsilon}-\f{\sg}{\gamma}\vp_{\sg,\epsilon}\right)$ is uniformly bounded in $L^2(\Omega_T)$. We also have the boundary conditions, $\nabla\left( n_{\sg,\epsilon}-\f{\sg}{\gamma}\vp \right)\cdot \nu =0$ and the conservation of both quantities. Therefore elliptic regularity theory gives us
\[
\| n_{\sg,\epsilon}-\f{\sg}{\gamma}\vp_{\sg,\epsilon} \|_{L^2(0,T;H^2(\Omega))} \le C.
\]
Therefore strong compactness  in space  holds for  the quantities $ n_{\sg,\epsilon}-\f{\sg}{\gamma} \vp$ and $\nabla[ n_{\sg,\epsilon}-\f{\sg}{\gamma}\vp]$.  Furthermore, from the limit~\eqref{eq:conv_svp}, it means  that both $ n_{\sg,\epsilon}$ and  $\nabla n_{\sg,\epsilon}$ are compact in space. 
%Their  compactness in time is an immediate consequence, thanks to the Lions-Aubin method, of the equation on $ n_{\sg,\epsilon}$ and of the bound~\eqref{eq:flux-est}.  
\AP{Compactness in time is also obtain for the quantity $n_{\sg,\epsilon}$ from \eqref{eq:control-deriv-n}. Again from Lions-Aubin lemma, we have the strong convergence \eqref{eq:strong-n-conv}. 
From this latter result and the strong convergence~\eqref{eq:conv_svp}, we obtain~\eqref{eq:strong-n-vp-conv}.}
The conclusion~\eqref{eq:weak-phi-conv} follows from this results. 

The bounds $0\leq n <1$ can be obtained as in the case $\e \to 0$ , see Theorem~\ref{th:existence} and we do not repeat the argument.
\\[5pt]
% -------------------------------------------
\textit{Step 3. Limiting equation.} Next, we need to verify that the limit of the subsequence $n_{\sg,\epsilon}$ satisfies the DCH equation.  The argument is different from the case $\e \to 0$ because we do not control $\nabla \vp_{\sg,\epsilon}$ in the case at hand. From the $L^2$ bound in \eqref{eq:flux-est}, we need to identify the weak limit 
\begin{equation}
J_{\sg,\epsilon}:= -B_{\epsilon}(n_{\sg,\epsilon})\nabla(\vp_{\sg,\epsilon} + \psi^\prime_{+,\epsilon}(n_{\sg,\epsilon})) \rightharpoonup -b(n)\nabla(\vp + \psi^\prime_{+}(n)) \quad \text{weakly in } L^2(\Omega_T). 
\label{eq:conv_J}\end{equation} 

For a test function  $\eta \in L^2(0,T;H^1(\Omega,\mathbb{R}^d))\cap L^\infty(\Omega_T,\mathbb{R}^d)$ and $\eta \cdot \mu = 0$ on $\partial \Omega \times (0,T)$, we  integrate the left-hand side to obtain
\[
\int_{\Omega_T} J_{\sg,\epsilon} \cdot \eta = -\int_{\Omega_T} \gamma \Delta \left(n_{\sg,\epsilon} - \f{\sg }{\gamma} \vp_{\sg,\epsilon}\right) \nabla \cdot (B_\epsilon(n_{\sg,\epsilon})\eta) + B_\epsilon(n_{\sg,\epsilon})\nabla \left(\psi^\prime_{+,\epsilon}(n_{\sg,\epsilon}) + \psi^\prime_-(n_{\sg,\epsilon}-\f{\sg}{\gamma}\vp_{\sg,\epsilon})\right) \cdot \eta. 
\]
We have mainly two types of terms on the right-hand side ~{$\int_{\Omega_T} \gamma \Delta \left(n_{\sg,\epsilon} - \f{\sg }{\gamma} \vp_{\sg,\epsilon}\right) \nabla \cdot (B_\epsilon(n_{\sg,\epsilon})\eta)$} and $\int_{\Omega_T} B_\epsilon(n_{\sg,\epsilon})\nabla \left(\psi^\prime_{+,\epsilon}(n_{\sg,\epsilon}) + \psi^\prime_-(n_{\sg,\epsilon}-\f{\sg}{\gamma}\vp_{\sg,\epsilon})\right) \cdot \eta $. 
Let us focus on the first term 
\[
\begin{aligned}
\int_{\Omega_T} \gamma \Delta \left(n_{\sg,\epsilon} - \f{\sg }{\gamma} \vp_{\sg,\epsilon}\right)\nabla \cdot (B_\epsilon(n_{\sg,\epsilon})\eta) =&  \int_{\Omega_T} \gamma \Delta \left(n_{\sg,\epsilon} - \f{\sg }{\gamma} \vp_{\sg,\epsilon}\right) B_\epsilon(n_{\sg,\epsilon}) \nabla \cdot \eta 
\\
&+ \int_{\Omega_T} \gamma \Delta \left(n_{\sg,\epsilon} - \f{\sg }{\gamma} \vp_{\sg,\epsilon}\right)   B^\prime_\epsilon(n_{\sg,\epsilon}) \nabla n_{\sg,\epsilon} \cdot\eta. 
\end{aligned}
\]
 From the strong convergence \eqref{eq:strong-n-conv} and the weak one \eqref{eq:conv_lapn} with the fact that $B_\epsilon(\cdot) \rightarrow b(\cdot)$ uniformly, we obtain the convergence of the first term of the right-hand side
\[
\int_{\Omega_T} \gamma \Delta \left(n_{\sg,\epsilon} - \f{\sg }{\gamma} \vp_{\sg,\epsilon}\right) B_\epsilon(n_{\sg,\epsilon}) \nabla \cdot \eta \rightarrow \int_{\Omega_T} \gamma \Delta n \; b(n) \nabla \cdot \eta,
\]
as $\sg,\epsilon \rightarrow 0$ and thus we have passed to the limit in the first term of the right hand side. For the second term, we
use that the derivative $B^\prime_\epsilon(\cdot) \to  b^\prime(\cdot)$ uniformly. We also use the strong convergence of $\nabla n_{\sg,\epsilon}$ from \eqref{eq:strong-n-conv}.
From the results above and a generalized version of the Lebesgue dominated convergence theorem we obtain
\[
\int_{\Omega_T} \gamma \Delta \left(n_{\sg,\epsilon} - \f{\sg }{\gamma} \vp_{\sg,\epsilon}\right)   B^\prime_\epsilon(n_{\sg,\epsilon}) \nabla n_{\sg,\epsilon} \cdot\eta   \rightarrow \int_{\Omega_T} \gamma \Delta n   b^\prime(n) \nabla n \cdot\eta , 
\]
as $\sg,\epsilon \rightarrow 0$. 

Let us now pass to the limit in $\int_{\Omega_T} B_\epsilon(n_{\sg,\epsilon})\nabla \left(\psi^\prime_{+,\epsilon}(n_{\sg,\epsilon}) + \psi^\prime_-(n_{\sg,\epsilon}-\f{\sg}{\gamma}\vp_{\sg,\epsilon})\right) \cdot \eta$. As in the case of the convergence $\epsilon \to 0$, we have that 
\[
    \int_{\Omega_T} B_\epsilon(n_{\sg,\epsilon})\nabla \left(\psi^\prime_{+,\epsilon}(n_{\sg,\epsilon}) \right)\cdot \eta,
\]
using the fact that ${B_\epsilon(\cdot) \psi^{\prime\prime}_{+,\epsilon}(\cdot) \rightarrow b(\cdot)\psi^{\prime\prime}_{+}(\cdot)}$ uniformly and the strong convergence \eqref{eq:strong-n-conv}. 
Since ${B_\epsilon(\cdot) \rightarrow b(\cdot)}$, we have ${\left(B_\epsilon\psi^{\prime\prime}_-\right)(\cdot) \rightarrow \left(b\psi^{\prime\prime}_-\right)(\cdot)}$.

\noindent Therefore, we pass to pass to the limit in ${\int_{\Omega_T} B_\epsilon(n_{\sg,\epsilon})\nabla \left(\psi^\prime_{-}(n_{\sg,\epsilon}-\frac{\sg}{\gamma}\vp_{\sg,\epsilon}) \right)\cdot \eta}$ using the convergence \eqref{eq:strong-n-vp-conv}. Altogether, we obtain the following convergence
\[
    \int_{\Omega_T} B_\epsilon(n_{\sg,\epsilon})\nabla \left(\psi^\prime_{+,\epsilon}(n_{\sg,\epsilon}) + \psi^\prime_-(n_{\sg,\epsilon}-\f{\sg}{\gamma}\vp_{\sg,\epsilon})\right) \cdot \eta \rightarrow \int_{\Omega_T} b(n)\nabla \left(\psi^\prime_{+}(n) + \psi^\prime_-(n)\right) \cdot \eta
\]
\par
This finishes the proof of \eqref{eq:conv_J}, i.e. that the limit solution $n$ satisfies the weak formulation of the DCH equation~\eqref{eq:CH}, and also the proof of Theorem~\ref{th:sgto0}.
\end{proof}
%-----------------------------------------
\section{Long-time behavior}
\label{sec:ltb}
%-----------------------------------------

To complete our study  of the RDCH model, we give some insights concerning the long-time behavior and convergence to steady states, $(n_\infty,\vp_\infty)$ determined by the steady problem
\begin{equation}
\begin{cases}
&
\nabla \cdot \left(b(n_\infty) \nabla \left(\vp_\infty + \psi_+^\prime(n_\infty) \right) \right) = 0 \quad \text{ in }\;  \Omega ,
\\[5pt]
&-\sg \Delta \vp_\infty + \vp_\infty = -\gamma \Delta n_\infty + \psi^\prime_-(n_\infty - \f{\sg}{\gamma} \vp_\infty )\text{ in } \Omega ,
\\[5pt]
&\f{ \p \left( n_\infty- \frac{\sigma}{\gamma}\vp_\infty\right)  }{\p \nu}  = b(n_\infty)\f{\p \left( \vp_\infty + \psi_+^\prime(n_\infty) \right) }{\p \nu} = 0 \quad \text{ on } \; \p \Omega.
\end{cases}
\label{eq:ststRDCH}
\end{equation}  

The analysis of the steady-states is not performed in this paper, however, numerical simulations can help us to have an idea of their shape for different initial situations. 

% Description of the steady states from a numerical point of view
The steady-states of the RDCH model present a configuration which minimizes the energy of the system. The solution obtained at the end of the simulation depends mainly on three parameters: the initial mass $M$, the width of the diffuse interface $\sqrt{\gamma}$ and the relaxation parameter $\sigma$. \par  
In fact, if the initial mass is large enough, saturated aggregates are formed and we can describe two regions in the domain: the aggregates and the absence of cells. Between these two regions, the transition is smooth and the length of this interface is $\sqrt{\gamma}$. 
If the initial mass is small, aggregates are still formed but they are thicker and their maximum concentration does not reach $1$ or the critical {value~$n^\star$} as in the definition of the potential \eqref{eq:pot1}. \par
The formation of aggregates happens only if $\gamma$ is small enough. If $\gamma$, the initial mass $M$ or the relaxation parameter $\sigma$ is too large, the solution converges to the constant one 
\[
   n_\infty = \frac{1}{|\Omega|}\int_\Omega n^0 dx, \quad \text{a.e. in } \Omega.
\]
A surprizing fact about these observations is that the long-time behavior of the solutions of the RDCH system seems to follow the analytical description of the steady-states made by Songmu \cite{songmu_asymptotic_1986}.    

%Furthermore, if we compare the numerical behavior of the RDCH system with the one of the DCH model, we observe that the only difference is that the relaxation breaks the symmetry and therefore the different aggregates, formed during the simulation, do not have the same maximum concentration.  \par

To  state our convergence result of the weak solutions of the RDCH model to steady-states, we consider a global weak solution  $(n,\vp)$ of the RDCH system with $\sg >0$, according to Theorem~\ref{th:existence}. The initial condition satisfies $0 \le n^0 < 1$ and has finite energy and entropy. so that we can use the a priori estimates from the transport structure, the energy and entropy dissipations  \eqref{eq:deriv-energy} and  \eqref{eq:entropy} (or \eqref {eq:energy_L}--\eqref{eq:entropy_L}), in particular
\begin{equation}
 0 \le n < 1 \text{ a.e. } (0, \infty)\times \Omega.
\label{eq:n-0-1}
\end{equation}
Based on the controls provided by these relations, and using a standard method, we are going to study the large time behavior as the limit for large $k$ of the sequence of functions 
\[
n_k(t,x) = n(t+k,x), \qquad \text{and}  \qquad \vp_k(t,x) = \vp(t+k,x).
\]
 %---------------------------------------------------------
\begin{proposition} [Long term convergence along subsequences]
Let $(n,\vp)$ be a weak solution of \eqref{eq:CH-relax}, \eqref{eq:CH-bound-relax} and initial condition $n^0$ with $0\le n^0 < 1$,  finite energy  and entropy. 
Then,  we can extract a subsequence, still denoted by index $k$,  of $(n_k, \vp_k)$ such that 
\begin{equation}
\lim_{k\to \infty} n_k(x,t) = n_{\infty}(x) , \quad \lim_{k\to \infty} \vp_k(x,t) = \vp_{\infty}(x)  \qquad \text{ strongly in } L^2\big((-T,T)\times \Omega\big), \quad  \forall T>0,
\label{eq:conv-nk}
\end{equation}
where $(n_{\infty},  \vp_{\infty})$ are solutions of \eqref{eq:ststRDCH} satisfying
\begin{equation}
 b(n_\infty) \nabla \left(\vp_\infty + \psi^\prime_+(n_\infty)\right) =0.
 \label{zeroflusststs}
 \end{equation}
\end{proposition}
%----------------------------------------------------------

\begin{proof} The proof uses the energy and entropy inequalities to obtain both uniform (in $k$) a priori bounds and zero entropy dissipation in the limit, which imply the result. We write these arguments in several steps.
\\
\noindent \textit{1st step.  A priori bounds from energy.} Energy decay implies that $\cae [n_k(t)]$ remains bounded in $k$ for $t> -k$. As a consequence,   the sequence $(n_k,\vp_k)$ satisfies
\begin{equation}
   \f{\sg}{2\gamma}  \int_\Omega \left|\vp_k(t)\right|^2 \leq \cae [n^0], \qquad \forall  t \geq 0,
   \label{eq:vp-est-time}
\end{equation}
\begin{equation}
   \f{\gamma}{2} \int_\Omega \left|\nabla \big(n_k(t)-\f{\sigma}{\gamma}\vp_k(t)\big) \right|^2 \leq \cae [n^0], \qquad \forall  t \geq 0,
   \label{eq:grad-est-time}
\end{equation}
\begin{equation}
\int_{-T}^T  \int_\Omega b(n_k) \big|\nabla (\vp_k + \psi'_{+}(n_k))\big|^2 := L_k(T) ,  \qquad  L_k(T) \to 0 \; \text{ as }\;  k \to \infty,
\label{eq:flux-est-time}
\end{equation}
and this last line is because 
\[
\int_{0}^\infty  \int_\Omega b(n) \big|\nabla (\vp+ \psi'_{+}(n))\big|^2 \le \cae [n^0], \quad  L_k(T) \leq \int_{k-T}^\infty  \int_\Omega b(n) \big|\nabla (\vp+ \psi'_{+}(n))\big|^2 ]  \underset{ k \to \infty }{\longrightarrow}  0 .
\]

\noindent  \textit{2nd step.  A priori bounds from entropy.} Because the right hand side in the entropy balance has a positive term (since $\psi''_-(n)\le 0, \forall n\in [0, 1]$), it cannot be used as easily as the energy. However, we can integrate~\eqref{eq:entropy_L} from $k-T$ to $k+T$, and, using the control of the negative term including $\psi_-$  as after~\eqref{eq:entropy_L},   we obtain the inequality 
\[
   \begin{aligned}
      \int_{-T}^T &\left[ \left| \Delta\left( n_k - \f{\sg}{\gamma} \vp_k \right)\right|^2 + \f{\sg}{\gamma}|\nabla \vp_k|^2  + \psi^{\prime\prime}_{+}(n_k)|\nabla n_k|^2\right] 
    \\ & \le \Phi[n(k-T)] - \Phi[n(k+T)] + \| \psi''_- \|_\infty  \left\|\nabla\left(n_k - \f{\sg}{\gamma}\vp_k \right)\right\|^2_{L^2((-T,T)\times \Omega)}
     \\ & \le \Phi[n(k-T)] - \Phi[n(k+T)] + \| \psi''_- \|_\infty    \frac{4T}{\gamma}    \cae [n(k-T)]
         \\ & \le \Phi[n(k-T)] - \Phi[n(k+T)] + \| \psi''_- \|_\infty    \frac{4T}{\gamma}    \cae [n^0].
   \end{aligned}
\] 
\noindent  \textit{3rd step.  Extracting subsequences.}
From these inequalities, we can extract subsequences of $(n_k, \vp_k)$ such that for $k \to \infty$, the following convergences hold toward some functions $n_\infty(x,t)$ and $ \vp_\infty(x,t)$.

We can conclude from inequalities \eqref{eq:vp-est-time} and the entropy control  that, as $k\to \infty$, 
\begin{equation}
\vp_k \rightharpoonup \vp_\infty \text{ weakly in } L^2\big(-T,T;H^1(\Omega)\big).
\label{eq:conv-weak-vp}
\end{equation} 
From the gradient bound  \eqref{eq:grad-est-time}, the $L^2$ bound in \eqref{eq:vp-est-time} and $0\leq n_k <1$, we obtain
\begin{equation}
n_k-\f{\sigma}{\gamma}\vp_k \rightharpoonup n_\infty - \f{\sigma}{\gamma}\vp_\infty \text{ weakly in } L^2\big(0,T; H^1(\Omega) \big),
\label{eq:weakcv1}
\end{equation}  
and thus
\begin{equation}
n_k \rightharpoonup n_\infty \text{ weakly in } L^2\big(0,T; H^1(\Omega) \big).
\label{eq:conv-weak-nk}
\end{equation}  
Finally, we obtain from \eqref{eq:flux-est-time} and the Cauchy-Schwarz inequality,
\begin{equation}
\p_t n_k \rightharpoonup \p_t n_\infty =0  \text{ weakly in } L^2\big(0,T; (H^1(\Omega))' \big) . 
\label{eq:conv-weak-dtnk}
\end{equation}  
Indeed,  for any test function $\phi \in C^\infty_0((-T,T)\times\Omega)$, it holds
\[
\int_{-T}^T \int_\Omega \p_t n_{k} \phi dx dt = -\int_{-T}^T \int_\Omega b(n_{k}) \nabla\left(\vp_{k} + \psi^\prime_+(n_{k})\right) \cdot \nabla \phi ,
\]
\[
\left| \int_{-T}^T \int_\Omega \p_t n_{k} \phi dx dt \right|^2 \leq  2T |\Omega| \| b\| _\infty \|  \nabla \phi \|_\infty^2
\int_{-T}^T \int_\Omega b(n_{k})\left| \nabla\left(\vp_{k} + \psi^\prime_+(n_{k})\right) \right|^2  \to 0 
\]
as $k \to \infty$. This also shows that $n_\infty$ only depends on $x$.
\par
\noindent \textit{4th step. Strong limits.}
The strong compactness  of $n_k$ and  $\vp_{k}$ follows from~\eqref{eq:grad-est-time} and the entropy control. Then, time compactness of $n_k$, stated in~\eqref{eq:conv-nk} follows from the Lions-Aubin lemma, thanks to~\eqref{eq:conv-weak-dtnk}.
The strong convergence of $\vp_{k}$ is a consequence of the elliptic equation for $\vp_{k}$ and of \eqref{eq:control-deriv-n-vp} which gives compactness in time of the quantity $n_k - \frac{\sg}{\gamma }\vp_k$.
And we also have, from the strong convergence of $n_k$ and \eqref{eq:weakcv1}, thanks to the above argument,
\begin{equation} 
b(n_k) \nabla\left(\vp_{k}+\psi^\prime_+(n_{k})\right)  \to  b(n_\infty) \nabla \left(\vp_\infty + \psi^\prime_+(n_\infty)\right)  = 0,
\label{eq:limit-conv}
\end{equation}
which establishes the zero-flux equality~\eqref{zeroflusststs}.
\end{proof}

%------------------------------------------------------
\section{Conclusion}
 The proposed relaxation system of the degenerate Cahn-Hilliard equation with single-well potential reduces the model to two parabolic/elliptic equations which can be solved by standard numerical solvers. The relaxation uses a regularization in space of the new unknown used to transform the original fourth-order equation into two second-order equations. This new system is a non-local relaxation of the original equation which is similar in a sense to the Cahn-Hilliard equation with a spatial interaction kernel derived in \cite{GL1, GL2}. 
 %However, unlike this model, we have been able to prove that in the limit of vanishing relaxation, we retrieve the original weak solutions of the DCH equation.
 We proved that in the limit of vanishing relaxation, we retrieve the original weak solutions of the DCH equation using compactness methods and estimates borrowed from energy and entropy functionals. The long-time behavior of the solutions of the RDCH system can also be studied along the same lines. We showed that a global solution of the system converges to a steady-state as time goes to infinity, with zero flux.
  
 The stationary states exhibit some interesting properties due to the degeneracy of the mobility. More precisely, they are split into two distinct zones: whether the mobility is zero, which is possible only in the pure phases, or the flux is null. 
 \par
The RDCH system aims at the design of a numerical method to simulate the DCH equation using only second order elliptic problems. Such a numerical scheme may depend on details of the relaxed model. For example, the solution represents a density and its numerical positivity is a desired property. Also, the discrete stability is useful and a change of unknown in the RDCH system might be better adapted, using $U=\vp-\frac{\gamma}{\sigma}n$,
\[
 \begin{aligned}
   \p_t n &= \nabla\cdot \left(b(n) \nabla\left(U +\f \gamma \sigma n + \psi_+^\prime(n)) \right) \right), \\
   -\sigma \Delta U + U &= -\f \gamma \sigma n +\psi^\prime_-(-\f \sigma \gamma U).
 \end{aligned}  
\] 
Even though this model also consists of a parabolic transport equation coupled with an elliptic equation, the regularity is enhanced. On the one hand, in  the first equation, the term  $\f \gamma \sigma n$ increases the diffusion for~$n$. On the other hand, the second equation regularizes for the new variable~$U$ because it depends on $n$ rather than $\Delta n$. 
%Altogether, this model improves the regularity of the two unknowns and might be interesting for numerical purposes. 
In a forthcoming work, we will propose a numerical scheme based on the RDCH system, that preserves the physical properties of the solutions.    
%------------------------------------------------------

%
%%%%%%%%%%%%%%%%%%%%%%%%%%%%%%%%%%%
%
%%%%%% BIBLIO %%%%%%%%%%%%%%%%%%%%%%
%
%%%%%%%%%%%%%%%%%%%%%%%%%%%%%%%%%%%%
%\pagestyle{myheadings}

%\begin{thebibliography}{99}
%\end{thebibliography}

\bibliographystyle{siam}  \bibliography{biblio.bib}

\end{document}